\documentclass[12pt]{article}
\usepackage{amsmath}
\usepackage{amssymb}
\usepackage{theorem}
\usepackage{enumerate} 
\usepackage{graphicx, tikz} 
   \usepackage{hyperref}

\usepackage{nameref}
\makeatletter
\let\orgdescriptionlabel\descriptionlabel
\renewcommand*{\descriptionlabel}[1]{%
  \let\orglabel\label
  \let\label\@gobble
  \phantomsection
  \edef\@currentlabel{#1}%
  \let\label\orglabel
  \orgdescriptionlabel{#1}%
}

\numberwithin{equation}{section}

\newtheorem{thm}{Theorem}[section]
\newtheorem{lemma}[thm]{Lemma}

\newtheorem{prop}[thm]{Proposition}
\newtheorem{cor}[thm]{Corollary}
{\theorembodyfont{\rmfamily}
\newtheorem{defn}[thm]{Definition}
\newtheorem{example}[thm]{Example}

\newtheorem{rmk}[thm]{Remark}
}

\newcommand{\R}{{\mathbb R}}
\newcommand{\Z}{{\mathbb Z}}
\newcommand{\C}{{\mathbb C}}
\newcommand{\D}{{\mathbb D}}
\newcommand{\E}{{\mathbb E}}

\newcommand{\PP}{{\mathbb P}}
\newcommand{\Q}{{\mathbb Q}}
\newcommand{\bbS}{{\mathbb S}}

\newcommand{\cA}{{\mathcal A}}
\newcommand{\cB}{{\mathcal B}}

\newcommand{\cE}{{\mathcal E}}

\newcommand{\cF}{{\mathcal F}}

\newcommand{\homega}{{\widehat \omega}}
\newcommand{\hPsi}{{\widehat \Psi}}
\newcommand{\hv}{{\hat v}}
\newcommand{\tPsi}{{\widetilde \Psi}}

\newcommand{\tH}{{\widetilde H}}
\newcommand{\tomega}{{\widetilde \omega}}

\newcommand{\tQ}{{\widetilde Q}}

\newcommand{\tu}{{\tilde u}}
\newcommand{\tv}{{\tilde v}}

\newcommand{\tZ}{{\widetilde Z}}
\newcommand{\ttau}{{\tilde \tau}}

\newcommand{\bF}{{\bar F}}
\newcommand{\bY}{{\bar Y}}

\newcommand{\bw}{{\bar w}}
\newcommand{\bT}{{\bar T}}
\newcommand{\bu}{{\bar u}}
\newcommand{\bv}{{\bar v}}
\newcommand{\bPsi}{{\bar{\Psi}}}
\newcommand{\oPsi}{{\overline{\Psi}}}
\newcommand{\btau}{{\bar{\tau}}}
\newcommand{\bR}{{\bar{R}}}
\newcommand{\bpi}{{\bar{\pi}}}
\newcommand{\bmu}{{\bar{\mu}}}
\newcommand{\bzeta}{{\bar{\zeta}}}

\newcommand{\infb}{{\textstyle{\inf_b}}}
\newcommand{\supb}{{\textstyle{\sup_b}}}

\renewcommand{\thesubsection}{\arabic{section}.\arabic{subsection}}

\newcommand{\qed}{\hfill \mbox{\raggedright \rule{.07in}{.1in}}}
 
\newenvironment{proof}{\vspace{1ex}\noindent{\bf
Proof}\hspace{0.5em}}{\hfill\qed\vspace{1ex}}
\newenvironment{pfof}[1]{\vspace{1ex}\noindent{\bf Proof of
#1}\hspace{0.5em}}{\hfill\qed\vspace{1ex}}

\newcommand{\eps}{{\epsilon}}
\newcommand{\Leb}{\operatorname{Leb}}
\newcommand{\Int}{\operatorname{Int}}

\newcommand{\supp}{\operatorname{supp}}
\newcommand{\diam}{\operatorname{diam}}
\newcommand{\sgn}{\operatorname{sgn}}

\title{Multivariate stable mixing local limit theorem for \\ intermittent diffeomorphisms and flows}

 \author{
 Ian Melbourne \thanks{Mathematics Institute, University of Warwick, Coventry, CV4 7AL, UK}
 \and
 {Fran\c{c}oise P\`ene}
 \thanks{Univ Brest, Universit\'e de Brest, LMBA,
 UMR CNRS 6205, 
 6 avenue Le Gorgeu, 29238 Brest cedex, France}
 \and
 Dalia Terhesiu
 \thanks{Mathematisch Instituut,
 University of Leiden, Niels Bohrweg 1, 2333 CA Leiden, Netherlands}
 }
 
\date{11 October 2026}

\begin{document}

 \maketitle
 
\begin{abstract}
We prove mixing local limit theorems for dynamical systems in the superdiffusive regime where there is convergence to a multivariate stable law (including the nonstandard domain of the central limit theorem).
We apply our results to $d$-dimensional observables for intermittent maps and solenoids, as well as intermittent semiflows and flows.
\end{abstract}

\tableofcontents

\section{Introduction}

Numerous statistical properties that are standard in the probability literature are known to hold for sufficiently chaotic dynamical systems and sufficiently smooth observables. For example, the central limit theorem (CLT) holds for H\"older observables for Axiom~A diffeomorphisms and flows~\cite{Bowen75,Ratner73,Ruelle78,Sinai60,Sinai72} and for large classes of nonuniformly expanding/hyperbolic systems~\cite{BunimovichSinaiChernov91,Liverani95,MN05,MT04,Young98,Young99}. Local versions of the CLT for discrete dynamical systems were obtained in~\cite{CoehloParry90,Gouezel05,RousseauEgele83}, but the continuous time case was more problematic and was only resolved relatively recently by Dolgopyat \& Nandori~\cite{DolgopyatNandori20}. (See also~\cite{AaronsonNakada17,DolgopyatNandori16,Iwata08}.)

There are many situations where the standard CLT fails and there is instead convergence to an $\alpha$-stable law $\alpha\in(1,2)$, see~\cite{Bruin23,BruinTerhesiuTodd21,Gouezel04,JungZhang18,Zweimuller03}.
Here the standard $n^{1/2}$ normalisation is replaced by superdiffusive $n^{1/\alpha}$ normalisation. 
There is also the borderline situation where the CLT holds with nonstandard normalisation, see $(n\log n)^{1/2}$, see~\cite{BalintChernovDolgopyat11,BalintGouezel06,Gouezel04,LMM26,LMMsub,SzaszVarju07,Zweimuller03}.
In particular, the CLT with nonstandard normalisation was proved for infinite horizon planar periodic dispersing billiards by Sz\'asz \& Varj\'u~\cite{SzaszVarju07} who also obtained local limit theorems for these examples. For further developments, see~\cite{PeneTerhesiu21}. Local CLTs with nonstandard normalisation for the corresponding Lorentz gas flow were recently obtained by P\`ene \& Terhesiu~\cite{PeneTerhesiu25}. Aaronson \& Denker~\cite{AD01b} had previously proved the nonstandard local CLT for Gibbs-Markov maps, and treated the stable law case too in~\cite{AD01}.

In this paper, we obtain results on local limit theorems in the stable law regime that are sufficiently general to include intermittent dynamics~\cite{PomeauManneville80} both in discrete and continuous time. Our method also covers the local CLT with nonstandard normalisation. Unlike in the afore-mentioned references, we allow observables with values in $\R^d$ for all $d\ge1$. 

We begin with some illustrative examples.
Proofs are given in Section~\ref{sec:PM}.

\begin{example} \label{ex:CFKM}
An explicit example of an intermittent map to which our theory applies is the map $T:M\to M$, $M=[0,1]$, given by
\begin{equation}
    \label{eq:CFKM}
    Tx = \begin{cases} 
        x(1+3^{1/\alpha} x^{1/\alpha}), & x \in[0, \frac{1}{3})  , \\
        3x-1, & x\in[\frac{1}{3}, \frac{2}{3})  , \\
        1 - (1-x) (1 + 3^{1/\alpha} (1-x)^{1/\alpha}), & x\in[ \frac{2}{3},1],
  \end{cases}
\end{equation}
where $\alpha\in(1,2]$. There is a unique absolute continuous $T$-invariant probability measure $\mu_M$.
This example is taken from~\cite{CFKM20}; the symmetry around $x=\frac12$ is irrelevant but we require that the exponents $1+1/\alpha$ at the neutral fixed points coincide.

Let $a_n=n^{1/\alpha}$ for $\alpha\in(1,2)$ and
$a_n=(n\log n)^{1/2}$ for $\alpha=2$.
Given $v:M\to\R^d$, we define $v_n=\sum_{j=0}^{n-1}v\circ T^j$.

\begin{thm} \label{thm:CFKM}
Let $T:M\to M$ be the map~\eqref{eq:CFKM} with $\alpha\in(1,2]$.
Let $v:M\to\R$ be a H\"older observable with $\int_M v\,d\mu_M=0$ and satisfying $v(0),\,v(1)\neq0$ and $v(0)\neq v(1)$.  It is known (see for example~\cite{CFKM20}) that $a_n^{-1}v_n$ converges in distribution to a two-sided $\alpha$-stable law $X_\alpha$; we denote the density of $X_\alpha$ by $g$.

 Under an aperiodicity assumption that holds typically,
\[
\lim_{n\to\infty}a_n\int_M \zeta_1\,\zeta_2\circ T^n\,\xi\circ(v_n-qa_n)\,d\mu_M
=g(q)\int_M \zeta_1\,d\mu_M \int_M \zeta_2\,d\mu_M \int_\R \xi\,d\Leb,
\]
for all $\zeta_1,\,\zeta_2\in C_c\big((0,1)\big)$, $\xi\in C_c(\R)$, $q\in\R$.
\end{thm}

The $\xi$ component encompasses the stable local limit theorem, whereas the test functions $\zeta_1$ and $\zeta_2$ represent mixing properties of $T$. Hence Theorem~\ref{thm:CFKM} is called a \emph{mixing local limit theorem (MLLT)}.
\end{example}

\begin{example}  \label{ex:Thaler}
We can extend to $C^\eta$ vector-valued observables $v:M\to\R^d$ of Thaler maps (see Section~\ref{sec:LSV}) with $k$ neutral fixed points $x_1,\dots,x_k$ all with exponent $1+1/\alpha$. We again assume that $v(x_i)\neq0$, $i=1,\dots,k$, and we require further that the vectors $\{(v(x_i),1),\, i=1,\dots,k\}$ span $\R^{d+1}$. (In particular, $k\ge d+1$.)
We also assume an aperiodicity assumption that holds typically.
Provided $d<\alpha^2$, we obtain a stable MLLT of the form
\[
\lim_{n\to\infty}a_n^d\int_M \zeta_1\,\zeta_2\circ T^n\,\xi\circ(v_n-qa_n)\,d\mu_M
=g(q)\int_M \zeta_1\,d\mu_M \int_M \zeta_2\,d\mu_M \int_{\R^d} \xi\,d\Leb,
\]
for all $\zeta_1,\,\zeta_2\in C_c\big(M\setminus\{x_1,\dots,x_k\}\big)$, $\xi\in C_c(\R^d)$, $q\in\R^d$.
Here $g$ is the density of the $d$-dimensional $\alpha$-stable law limit for $n^{-1/\alpha}v_n$.

The assumption $d<\alpha^2$ can be 
removed provided $\eta\in(0,1]$ is sufficiently close to $1$ and a further technical condition is satisfied. For this it suffices that $v(x_i)$ has positive first coordinate for all $i=1,\dots,k$.
\end{example}

\begin{example} \label{ex:ex}
Our results go over almost word for word to nonuniformly expanding semiflows built over intermittent maps of the types in Examples~\ref{ex:CFKM} and~\ref{ex:Thaler}.
See Section~\ref{sec:intflow}. The only differences are that the neutral fixed points become neutral periodic orbits and discrete time $n$ is changed to continuous time $t$.

Also, our results extend almost word for word to uniformly hyperbolic intermittent invertible maps and flows provided that we assume exponential contraction along stable leaves, see Sections~\ref{sec:sol} and~\ref{sec:intflow}.
However, when the contraction along stable leaves is subexponential, our results hold only for a restricted class of H\"older observables.
\end{example}

\begin{rmk}
All results above are stated for H\"older observables but, as usual, they apply equally well to dynamically H\"older observables.
\end{rmk}

\begin{rmk}
For brevity, we have focussed on nonlattice local limit theorems in this introduction. In Section~\ref{sec:absYT} and later sections, we also consider observables $v:M\to\Z^d$ as well as observables that are cohomologous to observables with values in a lattice.
\end{rmk}

We end this introduction by discussing several caveats:
\begin{itemize}
\item[(i)]  In the examples above, we require that the number $k$ of neutral fixed points (or periodic orbits) exceeds the dimension $d$.
This is due to a nondegeneracy assumption, Definition~\ref{defn:nondeg}.
(In Theorem~\ref{thm:CFKM}, this is the condition that $v(0)\neq v(1)$;
in Example~\ref{ex:Thaler}, it is the spanning assumption on $(v(x_i),1)$.)

We expect that it is a straightforward exercise to relax the nondegeneracy assumption and hence to have no restriction on the number of neutral fixed points, resolving this caveat.
\item[(ii)] 
If $d\ge\alpha^2$, we require further technical assumptions as in Example~\ref{ex:Thaler}.

We note that this caveat did not arise in previous work~\cite{DolgopyatNandori20,PeneTerhesiu25} which restricted to the cases $d=1$ and $d=\alpha=2$.
\item[(iii)]  There are technical issues in extending results from nonuniformly expanding maps/semiflows to their invertible nonuniformly hyperbolic counterparts~\cite[Section~8]{BBM19} as indicated in Example~\ref{ex:ex}. 

These issues are easily resolved when there is exponential contraction along stable leaves as is assumed in~\cite{DolgopyatNandori20}, or when the observables of interest are effectively constant along stable leaves as in~\cite{PeneTerhesiu25}. In general, the situation is unresolved even for the much easier standard CLT regime considered in~\cite{DolgopyatNandori20}.

\item[(iv)] The results in this paper rely heavily on the existence of a \emph{first} return map that is uniformly expanding/hyperbolic with bounded distortion. Many examples~\cite{BalintChernovDolgopyat11,BalintGouezel06,Bruin23,BruinTerhesiuTodd21,DolgopyatNandori20,LMMsub,PeneTerhesiu25,SzaszVarju07} require further (non-first return) inducing to obtain uniformly expansion/hyperbolicity with bounded distortion. Reinducing creates lifting issues for test functions $\zeta_1$, $\zeta_2$.

Such issues for mixing local central limit theorems are dealt with in~\cite[Proposition~4.3]{DolgopyatNandori20}. The corresponding result in the superdiffusive regime is more subtle and is the subject of work in progress.
\end{itemize}

The remainder of this paper is organised as follows:

In Section~\ref{sec:abstract}, we state and prove an abstract result, Theorem~\ref{thm:abstract}, on the stable MLLT for suspension flows.
Section~\ref{sec:rv} contains preliminary material on regularly varying vector-valued observables. 
In Section~\ref{sec:GM}, we recall the notion of Gibbs-Markov semiflow, and show how to verify the hypotheses of Theorem~\ref{thm:abstract} for such systems.
Section~\ref{sec:GM2} does the same for Gibbs-Markov flows.
In Section~\ref{sec:YT}, we consider the analogous questions for discrete time dynamical systems.

In Section~\ref{sec:NUE}, we use our abstract results to obtain stable MLLTs for nonuniformly expanding maps and semiflows.
Then we consider applications to intermittent dynamical systems in Section~\ref{sec:PM}.

\vspace{2ex}

\paragraph{Notation}
We use ``big O'' and $\ll$ notation interchangeably, writing $a_n=O(b_n)$ or $a_n\ll b_n$
if there are constants $C>0$, $n_0\ge1$ such that
$a_n\le Cb_n$ for all $n\ge n_0$.
We write $a_n\approx b_n$ if $a_n\ll b_n \ll a_n$.
As usual, $a_n=o(b_n)$ means that $\lim_{n\to\infty}a_n/b_n=0$
and $a_n\sim b_n$ means that $\lim_{n\to\infty}a_n/b_n=1$.

For $X$ a topological space, let $C_c(X)$ denote the space of continuous functions supported in a compact subset of $X$.
Given a measurable subset $E\subset\R^m$, we write $|E|$ to denote the Lebesgue measure of $E$. We denote the open unit ball of radius $h$ centred at $x$ by $B_h(x)$.

Let $\bbS^{m-1} = \{x \in \R^m : |x| = 1 \}$ denote the unit sphere in $\R^m$,
where $|x|$ denotes the Euclidean norm of $x$.

Given $g:\R^m\to\R$, we define the Fourier transform $\hat g(x)=(2\pi)^{-m}\int_{\R^m}e^{-is\cdot x}g(s)\,ds$.

\section{Abstract stable mixing local limit theorem for flows}
\label{sec:abstract}

In this section, we obtain an abstract result concerning the stable mixing local limit theorem (stable MLLT) for suspension flows. 

\subsection{Abstract setting and hypotheses}
\label{subsec:abstract}

Our abstract setting is the following.
Let $(Y,\mu)$ be a probability space and let $F:Y\to Y$ be
a measure-preserving transformation.
Let $\tau:Y\to\R^+$ be an integrable roof function with $\btau=\int_Y \tau\,d\mu$.
We form the suspension 
\[
Y^\tau=\{(y,u)\in Y\times\R:0\le u\le\tau(y)\}/\sim
\quad\text{where}\quad (y,\tau(y))\sim (Fy,0)\},
\]
with probability measure $\mu^\tau=(\mu\times\rm{Lebesgue})/\btau$.
Define the measure-preserving suspension flow (or semiflow)
$F_t:Y^\tau\to Y^\tau$ given by
$F_t(y,u)=(y,u+t)$ computed modulo identifications.

For $d\ge1$, let $w:Y^\tau\to\R^d$ be an integrable observable with $\int_{Y^\tau}w\,d\mu^\tau=0$.
Define the induced observable
\[
W:Y\to\R^d, \qquad W(y)=\int_0^{\tau(y)}w(y,u)\,du,
\]
and the pair 
\[
\Psi=(W,\ttau):Y\to\R^{d+1}, \quad \text{where} \quad
\ttau=\tau-\btau.
\]
Then $\int_Y \Psi\,d\mu=0$.
As usual, we write $\Psi_n=\sum_{j=0}^{n-1}\Psi\circ F^j$
and similarly for $W_n,\,\tau_n, \ttau_n$.

Fix $\alpha\in(1,2]$ and let $\ell:[0,\infty)\to(0,\infty)$ be slowly varying\footnote{A measurable function $\ell:[0,\infty)\to(0,\infty)$ is \emph{slowly varying} if $\lim_{t\to\infty}\ell(\lambda t)/\ell(t)=1$ for all $\lambda>0$.}.
Define $\tilde\ell=\ell$ for $\alpha<2$ and
$\tilde\ell(x)=1+\int_1^{1+x}\frac{\ell(u)}{u}\,du$ for $\alpha=2$. Let $a_n>0$ satisfy
$a_n^\alpha\sim n\tilde\ell(a_n)$ as $n\to\infty$.

Fix a collection $\cA_0$ of measurable subsets of $Y$.
Let $\cA$ denote the collection of rectangles of the form
 $A=A_0\times I\subset Y^\tau$ where $A_0\in\cA_0$ and 
$I\subset [0,\inf_{A_0}\tau]$ is an interval.

Our main hypotheses are as follows:
\begin{description}
\item[(N1)\label{itm:N1}] \textbf{Joint stable MLLT for $\Psi=(W,\ttau)$.}
We suppose that there is a 
continuous density $g_{d+1}$ for a $(d+1)$-dimensional random vector with the following property:

Let $A_0,B_0\in \cA_0$. Let $K\subset\R^{d+1}$ be bounded with $|\partial K|=0$, and let $L>0$.  Then 
\begin{align*}
\lim_{n\to\infty}
\sup_{\genfrac{}{}{0pt}{}{z\in\R^{d+1}}{|z|\le L}}
\Big|
a_n^{d+1}\mu\big(y\in A_0\cap F^{-n} B_0 :\,  \Psi_n &(y)\in za_n +K\big)
\\[-2.75ex] & - g_{d+1}(z)\mu(A_0)\mu(B_0)|K|\Big|=0 .
\end{align*}
\item[(N2)\label{itm:N2}] \textbf{Joint LLD for $\Psi=(W,\ttau)$.}
For any $h,\,r_0>0$, there exist $r\in(0,r_0)$, $C>0$ such that
\[
\mu(\Psi_n\in B_h(z))\le C \frac{n}{a_n^{d+1}}\, \frac{\tilde\ell(|z|)}{|z|^\alpha}\Big(1+ \Big(\frac{|z|}{a_n}\Big)^r\Big)
\quad\text{for all $n\ge1$, $z\in\R^{d+1}$.}
\]
\item[(N3)\label{itm:N3}] \textbf{Extra assumption for $d\ge\alpha^2$.} 
When $d\ge\alpha^2$, set $\gamma_0=(d-\alpha^2)/(d+1-2\alpha)\in[0,1)$.\footnote{Note that $d-\alpha^2<d+1-2\alpha$ for all $d\ge1$, $\alpha>1$.}
We require that there exists $\gamma>\gamma_0$ such that
for any $q\in\R^d$ and bounded sets
$K_1\subset\R^{d+1}$, $K_2\subset\R$, 
\begin{align*}
\sum_{0\le n\le t^\gamma}\mu^\tau\Big\{(y,u)\in Y^\tau: \,  \Psi_n(y) & \in (qa_t,t-n\btau)+K_1,
\\  \tau_n(y) & \le u+t<\tau_{n+1}(y),\,u\in K_2\Big\}
=o(a_t^{-d}).
\end{align*}
\end{description}

Define
\begin{equation} \label{eq:H}
H:Y^\tau\to\R^d, \qquad H(y,u)=\int_0^u w(y,t)\,dt.
\end{equation}

\begin{description}
\item[(N4)\label{itm:N4}] \textbf{Coboundary}
For any $A\in\cA$ and any $\delta>0$, there exist 
disjoint subsets $A_1,\dots,A_n\subset A$ with $A_1,\dots,A_n\in\cA$ such that
\[
\mu^\tau\Big(A\setminus \bigcup_{j=1}^nA_j\Big) < \delta
\quad\text{and}\quad 
\max_{1\le j\le n}\, \diam H(A_j)< \delta .
\]
In addition, $H(A)$ is a bounded subset of $\R^d$ for each $A\in\cA$.
\end{description}

\begin{rmk} 
(a) Hypothesis \ref{itm:N1} is ``fully nonlattice'': the roof function $\tau$ is $\R$-valued, the observable $w$ is $\R^d$-valued, and \ref{itm:N1} is a nonlattice MLLT for the induced joint observable $\Psi$ which has values in $\R^{d+1}$.
In Section~\ref{sec:YT}, we consider fully lattice and hybrid results for maps where $\tau$ is $\Z$-valued.

\vspace{1ex} \noindent (b)
Hypothesis~\ref{itm:N2} can be relaxed slightly without weakening any of the results. 
Let $\ell_1$ be a slowly varying function.
It suffices that there exists $D>0$ such that \ref{itm:N2} holds as stated for $|z|\le Dn$ and holds with $\tilde\ell$ replaced by $\ell_1$ for $|z|>Dn$. 

\vspace{1ex} \noindent (c)
The dimension restriction in \ref{itm:N3} is somewhat unsatisfactory but the restriction $d\le 2$ is already present in the simpler setting of~\cite{DolgopyatNandori20}, see~\cite[Lemmas~4.3 and~4.6]{DolgopyatNandori20}, and also in~\cite{PeneTerhesiu25}.

\vspace{1ex} \noindent (d)
Assumptions \ref{itm:N1}--\ref{itm:N3} enable us to prove a stable MLLT for the flow modulo a coboundary $H\circ F_t-H$. Dealing with the coboundary is generally considered to be a ``folklore'' result, though the details are often vague  and rely on certain non-explicit topological/metric properties of the underlying probability space (separability, completeness, local compactness). 
Also, it is usually required that $H$ is continuous and bounded, which is not the case in our main examples.

Assumption~\ref{itm:N4} enables us to remain in a purely probabilistic setting in the formulation of Theorem~\ref{thm:abstract}, and the setting is sufficiently flexible for our applications.
\end{rmk}

We can now state our stable MLLT for $w:Y^\tau\to\R^d$. Let $w_t=\int_0^t w\circ F_s\,ds$.
Let $g_{d+1}$ be the density in hypothesis~\ref{itm:N1}, and
define the $d$-dimensional density 
\begin{equation} \label{eq:gd}
g(q)=\btau^{d/\alpha}g_d(q\btau^{1/\alpha}) 
\quad\text{where}\quad
g_d(q)=\int_{-\infty}^\infty g_{d+1}(q,y)\,dy.
\end{equation}

 \begin{thm}[Stable MLLT for $w$]
 \label{thm:abstract}
Assume \ref{itm:N1}--\ref{itm:N4}.
 Let $A, B\in\cA$.
 Let $K\subset\R^d$ be bounded with $|\partial K|=0$ and let $q\in\R^d$. 
Then 
\[
  \lim_{t\to\infty}a_t^d\mu^\tau\Big(x\in A\cap F_t^{-1} B:\, w_t(x)\in qa_t+K\Big)
  = g(q)\mu^\tau(A)\mu^\tau(B)|K|.
 \]
\end{thm}

The remainder of this section is concerned with the proof of Theorem~\ref{thm:abstract}.
In Subsection~\ref{sec:pf1}, we prove the result ``modulo a coboundary'' using assumptions \ref{itm:N1}--\ref{itm:N3}. 
In Subsection~\ref{sec:pf2}, we show how to remove the coboundary using assumption~\ref{itm:N4}.

\subsection{Proof of Theorem~\ref{thm:abstract} up to a coboundary}
\label{sec:pf1}

The suspension flow $F_t:Y^\tau\to Y^\tau$ satisfies
 \begin{equation} \label{eq:lap}
  F_t(y,u)=(F^{N_t(y,u)}y, u+t-\tau_{N_t(y,u)}(y)),
 \end{equation} 
 where the \emph{lap number} $N_t(y,u)$ is given by
\[
  \tau_{N_t(y,u)}(y)\le u+t<\tau_{N_t(y,u)+1}(y).
\]
Note that
\[
w_t(y,u)= W_{N_t(y,u)}(y)+H\circ F_t(y,u)-H(y,u),
\]
where $H$ was defined in~\eqref{eq:H}.

 \begin{lemma}
 \label{lem:abstract}
Assume \ref{itm:N1}--\ref{itm:N3}.
 Let $A, B\in\cA$.
 Let $K\subset\R^d$ be bounded with $|\partial K|{=0}$ and let $q\in\R^d$. Then writing $x=(y,u)\in Y^\tau$,
\[
  \lim_{t\to\infty}a_t^d\mu^\tau\Big(x\in A\cap F_t^{-1} B:\, W_{N_t(x)}(y)\in qa_t+K\Big)
  = g(q)\mu^\tau(A)\mu^\tau(B)|K|.
 \]
\end{lemma}

\begin{proof}
The strategy of the proof 
 follows those of~\cite[Lemma 4.3]{AaronsonNakada17},~\cite[Theorem 1]{AaronsonTerhesiu20} and~\cite[Theorem 3.1]{DolgopyatNandori20}.

Write $A=A_0\times I,\,B=B_0\times J\subset \cA$
 and set $K_u=K\times(u-J)$.
Let $\omega_t=
  \mu^\tau\big( x \in A\cap F_t^{-1}B:\, W_{N_t(x)}(y)\in qa_t+K\big)$.
For $x=(y,u)\in Y^\tau$, it follows from~\eqref{eq:lap} that
\begin{align*}
\omega_t & = \mu^\tau\big(y\in A_0\cap F^{-N_t(x)}B_0,\,u\in I:\,u+t-\tau_{N_t(x)}(y)\in J,\,W_{N_t(x)}(y)\in qa_t+K\big) \\
  &=\sum_{m\ge0}\mu^\tau\big(y\in A_0\cap F^{-m}B_0,\, u\in I:\, 
\Psi_m(y)  \in (qa_t,t-m\btau)+K_u \big).
 \end{align*}
Let $\gamma\in(\gamma_0,1)$ be as in \ref{itm:N3} for $d\ge \alpha^2$ and $\gamma=0$ otherwise.
By \ref{itm:N3}. 
\[
  \omega_t
 = \sum_{m>t^\gamma}\mu^\tau\big(y\in A_0\cap F^{-m}B_0,\, u\in I:\, 
\Psi_m(y)  \in (qa_t,t-m\btau) +K_u \big) + o(a_t^{-d})
\]
Hence
\[
  \btau\omega_t = 
\sum_{m>t^\gamma} \int_I Q_m(t,u)\,du
+ o(a_t^{-d})
\]
where
\[
Q_m(t,u)=\mu\big(y\in A_0\cap F^{-m}B_0:\,
\Psi_m(y)  \in (qa_t,
t-m\btau)+K_u\big).
\]

Fixing $L>0$, we can write
\[
\btau  \omega_t= S_1(t,L)+S_2(t,L)+o(a_t^{-d})
\]
where
\[
  S_1(t,L)=\sum_{\stackrel{m>t^\gamma:}{|m-t/\btau|\le La_t}}\int_I Q_m(t,u)\,du, \qquad
  S_2(t,L)=\sum_{\stackrel{m>t^\gamma:}{|m-t/\btau|> La_t}}\int_I Q_m(t,u)\,du.
\]

Thus, it suffices to show:
\begin{align} \label{eq:S1}
& \lim_{L\to\infty}\lim_{t\to\infty} a_t^d S_1(t,L)=
 \btau g(q)\mu^\tau(A)\mu^\tau(B)|K|, \\
& \lim_{L\to\infty}\limsup_{t\to\infty} a_t^d S_2(t,L)=0.
\label{eq:S2}
\end{align}

In~\eqref{eq:S1}, we are in the range $|m-t/\btau|\le La_t$.
It follows that
$a_t\sim \btau^{1/\alpha}a_m$ and $z_{m,t}=\Big(\dfrac{qa_t}{a_m},\dfrac{t-m\btau}{a_m}\Big)$ is bounded.
By \ref{itm:N1},
\begin{align*}
a_m^{d+1}Q_m(t,u) & =
g_{d+1}(z_{m,t})\mu(A_0)\mu(B_0)|K_u|+h_m(t,u)
\\ &  =
\btau g_{d+1}(z_{m,t})\mu(A_0)\mu^\tau(B)|K|+h_m(t,u)
\end{align*}
where $\lim_{t\to\infty}h_m(t,u)=0$ uniformly in $|m-t/\btau|\le La_t$ and $u\in I$.
Hence
\[
a_t^d\int_I Q_m(t,u)\,du  
=\btau^{2+(d+1)/\alpha} a_t^{-1} g_{d+1}\big(q\btau^{1/\alpha},(t/\btau-m)\btau^{1+1/\alpha} a_t^{-1}\big)\kappa+\tilde h_m(t)
\]
where $\kappa=\mu^\tau(A)\mu^\tau(B)|K|$ and
$\tilde h_m(t)=o(a_t^{-1})$ as $t\to\infty$ uniformly in $|m-t/\btau|\le La_t$.
It follows that
\begin{align*}
  a_t^d  S_1(t,L) 
 =\btau^{2+(d+1)/\alpha} a_t^{-1}\sum_{|j|\le La_t} 
g_{d+1}\big(q\btau^{1/\alpha},j\btau^{1+1/\alpha} a_t^{-1}\big)\kappa+o(1)
\end{align*}
as $t\to\infty$ for each fixed $L$.
Approximating Riemann integrals by Riemann sums with increments of width $\btau^{1+1/\alpha} a_t^{-1}$ yields
\[
\btau^{1+1/\alpha} a_t^{-1}
\sum_{|j|\le La_t}
g_{d+1}\big(q\btau^{1/\alpha},j\btau^{1+1/\alpha} a_t^{-1}\big)
 \to  
\int_{-L\btau^{1+1/\alpha}}^{L\btau^{1+1/\alpha}} g_{d+1}(q\btau^{1/\alpha},y)\,dy
\]
as $t\to\infty$.
Hence
\[
\lim_{L\to\infty} \lim_{t\to\infty}a_t^dS_1(t,L)
  = \btau^{1+d/\alpha} g_d(q\btau^{1/\alpha})\, \kappa
  = \btau g(q)\, \kappa
\]
by definition of $g_d$ and $g$ in~\eqref{eq:gd}.
This completes the verification of~\eqref{eq:S1}.

\vspace{1ex}
Next, $Q_m(t,u)\le 
\sup_{u\in I} \mu(\Psi_m\in \big(qa_t,t-m\btau)+K_u\big)$.
In~\eqref{eq:S2}, we are
in the range $a_t\ll |m-t/\btau|$. Hence $|(qa_t,t-m\btau)|\approx |t-m\btau|$ and it follows from~\ref{itm:N2} that
\[
Q_m(t,u)\ll \tQ_m(t) \quad\text{where}\quad
 \tQ_m(t)=\frac{m}{a_m^{d+1}}\, \frac{\tilde\ell(|t-m\btau|)}{|t-m\btau|^\alpha}\Big(1+\Big(\frac{|t-m\btau|}{a_m}\Big)^r\Big)
\]
for any $r>0$ sufficiently small.
Hence~\eqref{eq:S2} is an immediate consequence of the next three results.

\begin{prop} \label{prop:1}
\(
\lim_{L\to\infty}\limsup_{t\to\infty}a_t^d\sum_{La_t<| m-t/\btau|<ct}\tQ_m(t)=0
\)
for all $c\in(0,1/\btau)$.
\end{prop}

\begin{proof}
By the choice of $c$, in this range, $m\approx t$.
Choose $r\in(0,\alpha-1)$. Then
\begin{align*}
\sum_{La_t<  |m-t/\btau|<ct} \tQ_m(t) 
  & \ll \frac{t}{a_t^{d+1}}
\sum_{ m-t/\btau> L a_t}
\frac{\tilde\ell(m\btau-t)}{(m\btau-t)^\alpha} \Big(1+\Big(\frac{m\btau-t}{a_t}\Big)^r\Big) \\
  & \ll \frac{t}{a_t^{d+1+r}}
\sum_{ m-t/\btau> L a_t}
\frac{\tilde\ell(m\btau-t)}{(m\btau-t)^{\alpha-r}} \\
&  \approx 
\frac{t}{a_t^{d+1+r}} \frac{\tilde\ell(a_t)}{(La_t)^{\alpha-1-r}}
\approx \frac{1}{a_t^d} \frac{1}{L^{\alpha-1-r}}\frac{t\tilde\ell(a_t)}{a_t^{\alpha}}
\end{align*}
by Karamata. 
The result follows since $a_t^\alpha\sim t\tilde\ell(a_t)$.
\end{proof}

\begin{prop} \label{prop:2}
\(
\lim_{t\to\infty}a_t^d\sum_{m-t/\btau \ge ct}\tQ_m(t)=0
\)
for all $c>0$.
\end{prop}

\begin{proof}
In this range, $m\gg t$ so choosing $r\in(0,\alpha-1)$, 
\[
\sum_{ m-t/\btau\ge ct} \tQ_m(t) 
 \ll \frac{t}{a_t^{d+1}}
\sum_{ m-t/\btau\ge ct}
\frac{\tilde\ell(m\btau-t)}{(m\btau-t)^{\alpha-r}}
 \approx \frac{t}{a_t^{d+1}} \frac{\tilde\ell(t)}{t^{\alpha-1-r}}
 =  \frac{1}{a_t^d}\;\frac{\tilde\ell(t)}{t^{\alpha-2-r}a_t}.
\]
By Potter's bounds,
\[
  \frac{\tilde\ell(t)}{t^{\alpha-2-r}a_t}\ll \frac{t^s}{t^{\alpha+1/\alpha-2-r-s}}
=t^{-(\alpha+1/\alpha-2-r-2s)},
\]
for all $s>0$.
Since $\alpha+1/\alpha>2$, we can shrink $r,\,s$ so that $\alpha+1/\alpha>2+r+2s$. Then
  $\dfrac{\tilde\ell(t)}{t^{\alpha-2-r}a_t}=o(1)$
as required.
\end{proof}

\begin{prop} \label{prop:3}
\(
\lim_{t\to\infty}a_t^d\sum_{m\ge t^\gamma,\,m-t/\btau\le -ct}\tQ_m(t)=0
\)
for all $c>0$,
\end{prop}

\begin{proof}
Write the range as $t^\gamma\le m\le c't$ where
$c'=1/\btau-c$. (If $c'\le0$, then there is nothing to prove.)
Let $r,\,s>0$.
Since $t-m\btau\gg t$,
\begin{align*}
\sum_{t^\gamma\le m\le c't} \tQ_m(t)
\ll \sum_{t^\gamma\le m\le c't}\frac{m}{a_m^{d+1}}\,\frac{\tilde\ell(t)}{t^{\alpha-r}}
& \ll t^{-\alpha+r+s} \sum_{t^\gamma\le m\le c't}\frac{1}{m^{(d+1)/\alpha-s-1}}
\\ & \approx t^{-\alpha+r+s} 
\big\{(t^\gamma)^{2-(d+1)/\alpha+s}
+t^{2-(d+1)/\alpha+s}\big\}.
\end{align*}
Since $\alpha+1/\alpha>2$, we can shrink $r,\,s>0$ so that
\[
a_t^d\sum_{t^\gamma\le m\le c't} \tQ_m(t)\ll t^{d/\alpha-\alpha+r+2s}(t^\gamma)^{2-(d+1)/\alpha+s}+o(1).
\]

If $d<\alpha^2$, then $\gamma=0$ and we see immediately that $a_t^d\sum_{1\le m\le c't} \tQ_m(t)=o(1)$ for $r,\,s>0$ sufficiently small.

If $d\ge \alpha^2$, then the choice of $\gamma>\gamma_0$ ensures that
$d/\alpha-\alpha+\gamma(2-(d+1)/\alpha)<0$ and the result again follows.
\end{proof}

This completes the proof of the lemma.
\end{proof}

\subsection{Removing the coboundary $H\circ F_t-H$}
\label{sec:pf2}

\begin{pfof}{Theorem~\ref{thm:abstract}}
We apply Lemma~\ref{lem:vague}. (Uniformity in a parameter $\gamma$ is not required here.)
Let $\zeta_t(x)=W_{N_t(x)}(y)-qa_t$ and $a(t)=a_t^d/g(q)$. Then assumption~\eqref{eq:vague} is satisfied by Lemma~\ref{lem:abstract}.  
Let $H_1=-H$ and $H_2=H$. Then $H_1$ and $H_2$ satisfy property \ref{itm:H} in Appendix~\ref{app:vague} by \ref{itm:N4}.
By Lemma~\ref{lem:vague},
\[
\lim_{t\to\infty}a_t^d 
\mu^\tau (x\in A\cap F_t^{-1}B:\,
\zeta_t(x)+H_1(x)+H_2(F_tx)\in K) 
= \mu^\tau(A)\mu^\tau(B)|K|.
\]
But $\zeta_t(x)+H_1(x)+H_2(F_tx)=w_t(x)-qa_t$, so the result is proved.
\end{pfof}

\section{Preliminaries on regular variation}
\label{sec:rv}

In this section, we recall some basic definitions concerning regularly varying $\R^m$-valued random variables and we show how they are verified for the examples in this paper.

\begin{defn} \label{defn:rv}
An $\R^m$-valued random variable $X$ is
    \emph{regularly varying} with index $\beta\in\R$ if 
    there exists a Borel probability measure $\sigma$
    on
    $\bbS^{m-1}$, called the \emph{spectral measure}, such that
\[
        \lim_{t \to \infty}
        \frac{\PP(|X| > \lambda t, \ X / |X| \in E)}{\PP(|X| > t)}
        = \lambda^{\beta} \sigma(E)
\]
    for all $\lambda > 0$ and all Borel sets $E \subset \bbS^{m-1}$ with $\sigma(\partial E) = 0$.
\end{defn}

Let $X$ be integrable but not square integrable with $\E X=0$.
We suppose moreover that $X$ is regularly varying with index $\beta=-\alpha$ where $\alpha\in(1,2]$.
Taking $E=\bbS^{m-1}$, we see that $|X|$ is a scalar regularly varying
function with index $-\alpha$.
Hence there exists a slowly varying function $\ell:[0,\infty)\to(0,\infty)$
such that $\PP(|X|>t)=t^{\beta}\ell(t)$.
Define $\tilde\ell$ and $a_n$ as in Section~\ref{sec:abstract}.

For $\alpha\in(1,2)$, we have the following stable limit theorem~\cite{Rvaceva54}.
Suppose that $\E X=0$. If $X_1,X_2,\dots$ are independent copies of $X$, then 
$a_n^{-1}(X_1+\dots+X_n)\to_d X_\alpha$, where
$X_\alpha$ has characteristic function
\begin{equation} \label{eq:X}
\E\, e^{is\cdot X_\alpha}= \exp\Big\{-\int_{\bbS^{m-1}}|s\cdot x|^\alpha\Big(1-i\sgn(s\cdot x)
\tan\frac{\pi\alpha}{2}\Big)\cos\frac{\pi\alpha}{2}\,\Gamma(1-\alpha)\,d\sigma(x)\Big\}
\end{equation}
for $s\in\R^m$. 

For $\alpha=2$, the limiting random variable $X_2$ is a multivariate normal distribution with 
characteristic function
\begin{equation} \label{eq:X2}
\E\, e^{is\cdot X_\alpha}= \exp\Big\{-\int_{\bbS^{m-1}}|s\cdot x|^2
	\,d\sigma(x)\Big\} \quad\text{for $s\in\R^m$}.
\end{equation}

\begin{defn} \label{defn:nondeg}
	A regularly varying random variable $X$ with index $-\alpha$, 
	$\alpha\in(1,2]$, is \emph{nondegenerate} if
\[
\int_{\bbS^{m-1}}|s\cdot x|^\alpha\,d\sigma(x)>0 
\quad\text{for all $s\in\bbS^{m-1}$.}
\]
\end{defn}

Now we return to the setting of Section~\ref{sec:abstract} with roof function
$\tau:Y\to\R^+$ and observable $W:Y\to\R^d$.
We describe a general framework which encompasses the examples in this paper and where regular variation and nondegeneracy of the joint observable 
$\Psi=(W,\ttau):Y\to\R^{d+1}$ can be easily verified.

An intrinsic feature of our examples is
that the roof function $\tau:Y\to\R^+$ lies in $L^1$ but not $L^2$, and is regularly varying with index $-\alpha$, $\alpha\in(1,2]$. Moreover, there is a finite decomposition $Y=\dot\bigcup_{i=1}^k Y_i$ into sets of positive measure,
constants $c_i>0$, $i=1,\dots,k$, with $\sum c_i=1$,
and a slowly varying function $\ell$, such that 
\[
\mu(\tau 1_{Y_i}>t) \sim c_i \ell(t) t^{-\alpha} 
\quad\text{as $t\to\infty$}, \quad i=1,\dots,k.
\]
Furthermore, 
there exist constants $\omega_i\in\R^d$ and $Z\in L^p(Y,\R^d)$ for some $p>\alpha$ such that
\begin{equation} \label{eq:WZ}
W=\omega\tau+Z, \qquad \omega|_{Y_i}\equiv \omega_i.
\end{equation} 
Define $\tomega_i=(\omega_i,1)\in\R^{d+1}$ and
$\homega_i=\tomega_i/|\tomega_i|$, $i=1,\dots,k$.
Combining sets $Y_i$ with common values of $\omega_i$, we may suppose that $\homega_1,\dots,\homega_k$ are distinct.

\begin{prop} \label{prop:jointrv}
The joint observable $\Psi=(W,\ttau):Y\to\R^{d+1}$ is regularly varying with index $-\alpha$ and spectral measure
\[
\sigma_{d+1}=N^{-1}\sum_{i=1}^k c_i|\tomega_i|^\alpha\delta_{\homega_i}
\quad\text{where} \quad N=\sum_{i=1}^k c_i |\tomega_i|^\alpha.
\]
Moreover, $\Psi$ is nondegenerate if and only if the vectors $\tomega_1,\dots,\tomega_k$ span $\R^{d+1}$.
\end{prop}

\begin{proof} Write $\Psi=X+\tZ$ where
$X=\tomega \tau$, 
$\tomega|_{Y_i}\equiv \tomega_i$, $\tZ=(Z,-\btau)$.
Since $\tZ\in L^p$ with $p>\alpha$, it is easily seen that $\Psi$ is regularly varying if and only if $X$ is regularly varying, in which case they have the same index and spectral measure. Hence it suffices to verify the desired properties for $X$ instead of $\Psi$.

First, we show that $X$ is regularly varying with the given spectral measure $\sigma_{d+1}$. Fix $j\in\{1,\dots,k\}$ and let $E\subset\bbS^d$ be a Borel set containing $\homega_j$ such that $\homega_i\not\in E$ for $i\neq j$.
Since the vectors $\homega_i$ are distinct, 
it suffices to show that
\begin{equation} \label{eq:X1}
        \lim_{t \to \infty}
        \frac{\mu(|X| > \lambda t, \ X / |X| \in E)}{\mu(|X| > t)}
        = \lambda^{-\alpha}N^{-1} c_j|\tomega_j|^\alpha.
\end{equation} 
Now, $X/|X|=\sum_i 1_{Y_i}\homega_i$ and hence
$X(y)/|X(y)|\in E$ if and only if $y\in Y_j$.
It follows that
\begin{align*}
\mu(|X| > \lambda t, \ X / |X| \in E) & =
\mu(y\in Y_j:|X(y)| > \lambda t)
 =\mu(y\in Y_j:|\tomega_j|\tau(y) > \lambda t)
\\ &\sim c_j \ell(\lambda t/|\tomega_j|)|\tomega_j|^\alpha \lambda^{-\alpha}t^{-\alpha}
\sim c_j \ell(t)|\tomega_j|^\alpha \lambda^{-\alpha}t^{-\alpha}
\end{align*}
as $t\to\infty$ for all $\lambda>0$.
Similarly,
\begin{align*}
\mu(|X| > t) & = \sum_i \mu(y\in Y_i: |X(y)|>t)
 = \sum_i \mu(y\in Y_i: |\tomega_i|\tau(y)>t)
 \\ &\sim \sum_i c_i\ell(t/|\tomega_i|)|\tomega_i|^\alpha t^{-\alpha}
 \sim \sum_i c_i\ell(t)|\tomega_i|^\alpha t^{-\alpha}
=N\ell(t) t^{-\alpha},
\end{align*}
and~\eqref{eq:X1} follows.

It remains to verify the statement about nondegeneracy.
We have seen that $\sigma_{d+1}=\sum_{i=1}^k e_i\delta_{\homega_i}$ where $e_i>0$.
Hence $\int_{\bbS^d}|b\cdot x|^\alpha\,d\sigma_{d+1}(x)=
\sum_{i=1}^k e_i |b\cdot\homega_i|^\alpha$. Since $e_i>0$,
degeneracy reduces to the condition that 
$b\cdot\homega_1=\cdots=b\cdot\homega_k=0$ for some $b\in\bbS^d$.
Equivalently, $\homega_1,\dots,\homega_k$ do not span $\R^{d+1}$.
\end{proof}

\begin{rmk} Similarly, $W:Y\to\R^d$ is
regularly varying with index $-\alpha$ and spectral measure
\(
\sigma_d=\big(\sum_{i=1}^k c_i |\omega_i|^\alpha\big)^{-1}\sum_{i=1}^k c_i|\omega_i|^\alpha\delta_{\omega_i/|\omega_i|}.
\)
\end{rmk}

\section{Verification of \ref{itm:N1}--\ref{itm:N4} for Gibbs-Markov semiflows}
\label{sec:GM}

In this section, we consider hypotheses \ref{itm:N1}--\ref{itm:N4} from Section~\ref{sec:abstract} for the case of Gibbs-Markov semiflows. 
In Subsection~\ref{sec:GMflows}, we recall the notions of Gibbs-Markov maps and semiflows.
In Subsection~\ref{sec:NGM}, we
show that hypotheses \ref{itm:N1}, \ref{itm:N2}, \ref{itm:N4} hold for Gibbs-Markov semiflows under a regularity assumption on $\Psi$.
Hypothesis~\ref{itm:N3} is vacuously satisfied for $d<\alpha^2$ and is relaxed somewhat for Gibbs-Markov semiflows with $d\ge\alpha^2$ in Subsection~\ref{sec:relax}. An aperiodicity assumption required for \ref{itm:N1} is shown to be typical in Subsection~\ref{sec:aperiodic}.

\subsection{Gibbs-Markov maps and semiflows}
\label{sec:GMflows}

We begin by recalling the notion of Gibbs-Markov map.
Standard references include~\cite{AD01,ADU93}.
Let $(Y,\mu)$ be a probability space with an at
most countable partition $\beta$ and let $F:Y\to Y$ be a measure-preserving map. Define the \emph{separation time} $s(y,y')$ to be the least integer $n\ge0$ such that $F^ny$ and $F^ny'$ lie in distinct partition elements. It is assumed that 
$s(y,y')=\infty$ if and only if $y=y'$, so that $d_\theta(y,y')=\theta^{s(y,y')}$ is a metric for $\theta\in(0,1)$. 
Let $\xi=\frac{d\mu}{d\mu\circ F}:Y\to\R$. We say that $F$ is a  \emph{Gibbs-Markov map} if 
\begin{itemize}

\parskip=-2pt
\item[(i)] $F$ restricts to a measure-preserving bijection from $b$ onto a union of partition elements for each $b\in\beta$,
\item[(ii)]  $\inf_{b\in\beta}\mu(Fb)>0$,
\item[(iii)] There are constants $C>0$ and $\theta\in(0,1)$ such that
$|\log \xi(y)-\log \xi(y')|\le C d_\theta(y,y')$ for all $y,y'\in b$, $b\in\beta$.
\end{itemize}
The map $(F,Y,\mu)$ is mixing (and hence ergodic) if and only if for every $b_1,b_2\in\beta$ there exists $N\ge1$ such that $F^nb_1\cap b_2\neq\emptyset$ for all $n\ge N$.

Let $(F,Y,\mu)$ be a mixing Gibbs-Markov map and $\tau:Y\to\R^+$ an integrable roof function. 
We call $(F_t,Y^\tau,\mu^\tau)$ a \emph{Gibbs-Markov semiflow}.

\subsection{Verification of \ref{itm:N1}, \ref{itm:N2} and \ref{itm:N4}}
\label{sec:NGM}

As in Section~\ref{sec:abstract}, we are interested in observables $\Psi:Y\to\R^{d+1}$. We require the following regularity condition: there exist $\theta\in(0,1)$ and $C>0$ such that
\begin{equation} \label{eq:Psireg}
|\Psi(y)-\Psi(y')|\le C (\infb |\Psi|+1)d_\theta(y,y')
\quad\text{for all $y,y'\in b$, $\,b\in\beta$.}
\end{equation}
Increasing $\theta$ if necessary, we can assume that $\theta$ in the definition of Gibbs-Markov map and in~\eqref{eq:Psireg} coincide.

To formulate local limit theorems in the context of dynamical systems, we require one further definition. 

\begin{defn} \label{defn:aperiodic} 
Let $(Y,\mu)$ be a probability space and
let $F:Y\to Y$ be a measure-preserving map.
An observable $\Psi:Y\to\R^{d+1}$ is \emph{aperiodic} if 
the functional equation $\phi\circ F=\lambda e^{is\cdot \Psi}\phi$
has no measurable solutions 
$\phi:Y\to \bbS^1$ for all
$s\in\R^{d+1}\setminus\{0\}$, $\lambda\in\bbS^1$.
\end{defn}

\begin{lemma} \label{lem:GM}
Let $(F,Y,\mu)$ be a mixing Gibbs-Markov map.
Let $\Psi:Y\to\R^{d+1}$ be an integrable observable with $\int_Y\Psi\,d\mu=0$
and satisfying~\eqref{eq:Psireg}. 

Suppose that
$\Psi$ is regularly varying with index $-\alpha$,  $\alpha\in(1,2]$, and 
spectral measure $\sigma$ as in Definition~\ref{defn:rv}.
Then
\begin{itemize}
\item[(a)]
$a_n^{-1}\Psi_n\to_d X_\alpha$, where
$X_\alpha$ is as in~\eqref{eq:X} and~\eqref{eq:X2}.
\item[(b)]
Suppose further that $\Psi$ is nondegenerate. Then 
hypothesis \ref{itm:N2} is satisfied.
\item[(c)]
Suppose further that $\Psi$ is nondegenerate and aperiodic.
Let $\cA_0$ be the collection of measurable subsets $A_0\subset Y$ with $\mu(\partial A_0)=0$.
Then hypothesis \ref{itm:N1} is satisfied with
$g_{d+1}$ being the density for the random vector $X_\alpha$ 
in part~(a).
\end{itemize}
\end{lemma}

\begin{proof}
Part~(a) is standard and can be found in various forms in the literature, e.g.~\cite{AD01,AD01b,Gouezel04} and~\cite[Theorem~4.2]{CFKM20}.
We omit the details since the arguments are simpler than those for part~(c).
Parts~(b) and~(c) follow from Theorems~\ref{thm:LLD} and~\ref{thm:LLT} in Appendix~\ref{app:N1N2}.
\end{proof}

Let $Y^\tau$ be the suspension over $Y$ with integrable roof function $\tau:Y\to\R^+$.
As in Section~\ref{subsec:abstract},
we define $\cA$ to be the collection of rectangles $A=A_0\times I$ in $Y^\tau$ where $A_0\in\cA_0$ and $I\subset [0,\inf_{A_0}\tau]$ is an interval.

\begin{prop} \label{prop:N4}
Suppose that $(F,Y,\mu)$ is Gibbs-Markov and that $H:Y^\tau\to\R^d$ has the property that $H|_A:A\to\R^d$ is bounded and uniformly continuous on $A$ for each $A\in\cA$. Then $H$ satisfies hypothesis \ref{itm:N4}.
\end{prop}

\begin{proof}
Let $A=A_0\times I\in\cA$ and
$\delta>0$.
By uniform continuity of $H|_A$, there exists $N\ge1$ such that $|H(y,u)-H(y',u')|<\delta$
for all $y,y'\in A_0$, $u,u'\in I$ with $s(y,y')>N$ and $|u-u'|<1/N$.

Write $Y$ as a disjoint union of $N$-cylinders $C_1,C_2,\dots$ Since $\sum_{k=1}^\infty \mu(C_k){=1}$, there
exists $r\ge1$ such that
\[
\mu(Y\setminus {\textstyle\bigcup_{k=1}^r }C_k) < \btau \delta/|I|.
\]
Also, we can partition $I$ into a finite number of intervals $J_1,\dots, J_s$ with $|J_\ell|<1/N$.

Take $\{A_1,\dots,A_n\}=\{(A_0\cap C_k)\times J_\ell:1\le k\le r,\,1\le \ell\le s\}$.
By construction, $A_j\subset A$ and $\diam H(A_j)<\delta$ for each $j$.
Cylinder sets $C_k$ are open and closed, so $\partial C_k=\emptyset$.
Hence $\partial (A_0\cap C_k)\subset\partial A_0$ and it follows that $A_j\in\cA$ for each $j$.
Finally,
\begin{align*}
\mu^\tau\Big(A\setminus\bigcup_{j=1}^n A_j\Big)
 & =\mu^\tau\Big(A_0\times I\setminus \bigcup_{k=1}^r (A_0\cap  C_k)\times I\Big)
\\ & =\mu\Big(A_0\setminus \bigcup_{k=1}^r A_0\cap C_k\Big)|I|/\btau
 \le \mu\Big(Y\setminus \bigcup_{k=1}^r C_k\Big)|I|/\btau<\delta,
\end{align*}
completing the verification of \ref{itm:N4}.
\end{proof}

\subsection{Verification of \ref{itm:N3}}
\label{sec:relax}

Hypothesis \ref{itm:N3} is only required when $d\ge\alpha^2$.
In this subsection, we discuss how to verify this hypothesis. 
We continue to assume that $\Psi$ satisfies~\eqref{eq:Psireg}.

We assume the decomposition $W=\omega\tau+Z=\sum_i \omega_i\tau1_{Y_i}+Z$ in~\eqref{eq:WZ} with $\omega_i\in\R^d$ for $i=1,\dots,k$ and
$Z\in L^p$ for some $p>\alpha$.

\begin{prop} \label{prop:relax}
Suppose that $d\ge\alpha^2$ and define $\gamma_0$ as in \ref{itm:N3}.
Suppose that $Z\in L^2\cap L^p$ for some 
\[
p> \frac{2d}{\alpha}\,\frac{d+1-2\alpha}{d+\alpha^2-4\alpha+2}\,.\;
\footnote{Note that $\frac{d+1-2\alpha}{d+\alpha^2-4\alpha+2}<1$, so it suffices that $p>\frac{2d}{\alpha}$.}
\]
Assume that for some $\gamma\in(\gamma_0,1)$, there exists $c>0$ such that
\begin{equation} \label{eq:relax}
\mu\big(\min_{1\le m\le t^\gamma}|(\omega\tau)_m|/\tau_m < c\big) =o(a_t^{-d})
\quad\text{as $t\to\infty$}.
\end{equation}
Then \ref{itm:N3} holds.
\end{prop}

\begin{proof}
Let
$S_t=\sum_{1\le m\le t^\gamma}\mu^\tau(I_m)$ where
\[
I_m=
\Big\{(y,u)\in Y^\tau:|W_m(y)|\ll a_t,\,
\tau_m(y)\approx t,\,\tau_m(y)\le u+t<\tau_{m+1}(y),\,|u|\ll 1\Big\}.
\]
It suffices to show that $S_t=o(a_t^{-d})$.

Write $W=\omega\tau+Z=\omega\tau-\int_Y\omega\tau\,d\mu+\tZ$ where
$\tZ=Z-\int_Y Z\,d\mu$.
Then $\tZ_m=W_m-(\omega\tau)_m+m\int_Y\omega\tau\,d\mu$.
If $|(\omega\tau)_m|\ge c\tau_m$, then  
the conditions $|W_m|\ll a_t$, $\tau_m\approx t$ and $m\ll t^\gamma$ imply that $|\tZ_m|\approx t$ since $\gamma<1$. Hence $I_m\subset I_m'\cup I_m''$ where
\begin{align*}
I_m' & =
\Big\{(y,u)\in Y^\tau:|(\omega\tau)_m|< c\tau_m
,\,\tau_m(y)\le u+t<\tau_{m+1}(y),\,|u|\ll 1\Big\},
\\
I_m'' & =
\Big\{(y,u)\in Y^\tau: |\tZ_m(y)|\approx t
,\,\tau_m(y)\le u+t<\tau_{m+1}(y),\,|u|\ll 1\Big\}.
\end{align*}

The sets $\{\tau_m(y)\le u+t<\tau_{m+1}(y)\}$ are disjoint, so
\begin{align*}
\sum_{1\le m\le t^\gamma}\mu^\tau(I_m')
& 	\le \mu^\tau\Big(\min_{1\le m\le t^\gamma}|(\omega\tau)_m|/\tau_m < c,\,|u|\ll1\Big)
\\ & \ll
	\mu\Big(\min_{1\le m\le t^\gamma}|(\omega\tau)_m|/\tau_m < c\Big)
=o(a_t^{-d}).
\end{align*}

Similarly,
\[
\sum_{1\le m\le t^\gamma}\mu^\tau(I_m'')
\ll \mu\Big(\max_{1\le m\le t^\gamma}|\tZ_m|\gg t\Big).
\]
It follows from~\eqref{eq:Psireg} that $\tZ$ is piecewise H\"older.
Since  $\tZ\in L^2$,
standard moment estimates~\cite{M09b} for Gibbs-Markov maps yield
$\big|\max_{1\le m\le t}|\tZ_m|\,\big|_p\ll t^{1/2}$.
By the definitions
\[
p(1-\tfrac12 \gamma_0)=\tfrac12 p\, \frac{d+\alpha^2-4\alpha+2}{d+1-2\alpha}>\frac{d}{\alpha}.
\]
Hence we can shrink $\gamma>\gamma_0$ so that by Markov's inequality,
\[
\sum_{1\le m\le t^\gamma}\mu^\tau(I_m'')
\ll (t^\gamma)^{p/2}/t^p = t^{-p(1-\frac12\gamma)}=o(a_t^{-d}),
\]
completing the proof.
\end{proof}

We conjecture that the condition~\eqref{eq:relax} is typically satisfied, though it seems not so easy to prove.
Certainly it holds if there is a direction $e\in\bbS^{d-1}$ such that $\omega_1\cdot e,\dots,\omega_k\cdot e$ all have the same sign, which is an open and nonempty condition.

\subsection{Aperiodicity is typical}
\label{sec:aperiodic}

Let $(F,Y,\mu)$ be Gibbs-Markov.
We recall the notion of good asymptotics from~\cite{FMT07}.

\begin{defn}
Let $p_0\in Y$ be a fixed point for $F$ and let $p_n\in Y$, $n\ge1$, be a sequence of periodic points $F^np_n=p_n$.
The sequence of periodic points has \emph{good asymptotics} for $\Psi:Y\to\R^{d+1}$ if 
\begin{itemize}
\item
\(
\Psi_n(p_n)=n\Psi(p_0)+\kappa+J_n\gamma^n+o(\gamma^n)\)
as $n\to\infty$,
where $\gamma\in(0,1)$, $\kappa,\,J_n\in\R^{d+1}$.
\item For any $b\in \bbS^d$, it holds that
$b\cdot J_{n}=E_{n}\cos(n\theta+\theta_{n})$ where $E_{n}$ is a bounded sequence of real numbers with $\liminf_{n\to\infty}|E_{n}|>0$, and either (i) $\theta=\theta_{n}=0$ for all $n$, or
(ii) $\theta\in(0,\pi)$ and $\theta_{n}\in(\tilde\theta-\pi/12,\tilde\theta+\pi/12)$ for some $\tilde\theta$.
\end{itemize}
\end{defn}

\begin{prop} \label{prop:good} Assume that $\Psi:Y\to\R^{d+1}$ satisfies~\eqref{eq:Psireg}.
Let $p_n$, $n\ge1$, be a sequence of periodic points with good asymptotics for $\Psi$. Then $\Psi$ is aperiodic.
\end{prop}

\begin{proof}
Suppose that there exist $\phi:Y\to\bbS^1$ measurable, $\lambda\in\bbS^1$ and $s\in\R^{d+1}$ such that $\phi\circ F=\lambda e^{is\cdot\Psi}\phi$.
A Liv\v{s}ic regularity result for Gibbs-Markov maps~\cite[Theorem~1.1]{Gouezel06} (see also~\cite{BruinHollandNicol05}) guarantees that $\phi$ has a version that is continuous on partition elements.
By the definition of aperiodicity, Definition~\ref{defn:aperiodic},
we are required to show that $s=0$.

Note that $\phi\circ F^n=\lambda^n e^{is\cdot\Psi_n}\phi$.
In particular, we can evaluate $\phi$ at  $p_n=F^np_n$ to obtain
$\lambda^n e^{is\cdot\Psi_n(p_n)}=1$. 
Similarly, 
$\lambda e^{is\cdot\Psi(p_0)}=1$, and so 
$e^{is\cdot\{\Psi_n(p_n)-n\Psi(p_0)\}}=1$.
By good asymptotics,
\[
s\cdot (\kappa+J_n\gamma^n+o(\gamma^n))\equiv 0 \bmod 2\pi.
\]
Letting $n\to\infty$ yields $s\cdot\kappa=0\bmod 2\pi$ and 
hence
\[
s\cdot (J_n\gamma^n+o(\gamma^n))\equiv 0 \bmod 2\pi.
\]
For $n$ sufficiently large, we obtain 
$s\cdot (J_n\gamma^n+o(\gamma^n))\equiv 0$ and hence
$s\cdot J_n\to0$.
If $s\neq0$, then write
$s\cdot J_n=E_n\cos(n\theta+\theta_n)$ as in the definition of good asymptotics.
In case (i), we have $s\cdot J_n=E_n\to0$ contradicting the $\liminf$ condition on $E_n$. In case (ii), there is a subsequence $n_k$ with $\cos(n_k\theta+\theta_{n_k})$ bounded away from zero leading to the same contradiction.
Hence $s=0$ completing the proof.
\end{proof}

As shown in~\cite{FMT07}, good asymptotics holds for a $C^2$ open and $C^r$ dense of functions $\Psi:Y\to\R^{d+1}$.
We follow~\cite[Appendix~A]{MTtoral18} for an explanation specific to the Gibbs-Markov setting.
For convenience, suppose that $F:Y\to Y$ is full-branched (so $Fb=Y$ for each $b\in\beta$) and fix two partition elements $b_1$, $b_2$. We suppose that $b_0=b_1\cup b_2$ is a Riemannian manifold and that $\Psi|_{b_0}$ is $C^r$ for some $r\ge2$.
Starting with a fixed point $p_0\in b_1$ and a transverse homoclinic point 
in $b_2$, we can construct a sequence of
 periodic points $p_n\in b_1$, $p_n\to p_0$, with orbits lying in $b_1\cup b_2$.
By the construction,  good asymptotics holds except that possibly $\liminf_{n\to\infty}|E_n|=0$. By~\cite{FMT07}, $\liminf |E_n|$ depends continuously on $\Psi$ in the $C^2$ topology and is nonzero for a $C^r$ dense set of $\Psi$. Hence we obtain good asymptotics, and by Proposition~\ref{prop:good} aperiodicity, for a $C^2$ open and $C^r$ dense set of functions $\Psi$.

\section{Verification of \ref{itm:N1}--\ref{itm:N4} for Gibbs-Markov flows}
\label{sec:GM2}

In this section, we consider hypotheses \ref{itm:N1}--\ref{itm:N4} from Section~\ref{sec:abstract} for the case of Gibbs-Markov flows. 

Let $(Y,d)$ be a bounded metric space with Borel probability measure $\mu$
and let $F:Y\to Y$ be a measure-preserving transformation.
Let $(\bF,\bY,\bmu)$ be a Gibbs-Markov map
with partition $\beta=\{\bY_j\}$ and associated constant
$\theta\in(0,1)$.

We suppose that there is a measure-preserving semiconjugacy $\bar\pi:Y\to\bY$,
so $\bar\pi\circ F=\bF\circ\bar\pi$
and $\bar\pi_*\mu=\bmu$.
The separation time on $\bY$ lifts to a separation time on $Y$ given by
$s(y,y')=s(\bar\pi y,\bar\pi y')$ for $y,y'\in Y$.
Suppose that there exists $C>0$, $\theta\in(0,1)$ such that
\[
d(F^ny,F^ny')\le C(\theta^n+\theta^{s(y,y')-n})
\quad\text{for all $y,y'\in Y$, $n\ge1$}.
\]
Then we call $(F,Y,\mu)$ a \emph{two-sided Gibbs-Markov map}.

\begin{rmk}  The two-sided Gibbs-Markov maps that we have in mind are essentially the base of a two-sided Young tower as in~\cite{Young98} (except that we do not require full-branches for $\bF$). We have followed~\cite[Section~7.1]{BMT21} in omitting much of the structure often associated with Young towers, mentioning only those properties required in the sequel.
In particular, we have not mentioned a product structure, 
and
we have not made explicit the quotienting procedure (along local stable leaves) that passes from $F$ to $\bF$.

Of course, our class of two-sided Gibbs-Markov maps trivially includes the one-sided Gibbs-Markov maps considered in Section~\ref{sec:GM}. However,  the definition is flexible enough to include genuinely two-sided maps.
\end{rmk}

We consider \emph{Gibbs-Markov flows} $(F_t,Y^\tau,\mu^\tau)$, namely suspensions over a mixing two-sided Gibbs-Markov map $(F,Y,\mu)$ with integrable roof function $\tau:Y\to\R^+$.

Let $\cA_0$ be the collection of measurable subsets $A_0\subset Y$ with $\mu(\partial A_0)=0$ and let
$\cA$ consist of rectangles $A=A_0\times I\subset Y^\tau$.

\begin{lemma} \label{lem:GM2}
Let $(F_t,Y^\tau,\mu^\tau)$ be a Gibbs-Markov flow.
Given $w:Y^\tau\to\R^d$ integrable with $\int_{Y^\tau}w\,d\mu^\tau=0$,
define
$\Psi=(W,\ttau):Y\to\R^{d+1}$ and $H:Y^\tau\to\R^{d+1}$ as in Section~\ref{sec:abstract}.
Assume that $\Psi:Y\to\R^{d+1}$ is regularly varying with index $-\alpha$,  $\alpha\in(1,2]$. Suppose also
that $\Psi$ admits a decomposition of the form
\begin{equation} \label{eq:chi}
\Psi=\bPsi\circ\bar\pi+\chi\circ F-\chi,
\end{equation}
where $\bPsi:\bY\to\R^{d+1}$ satisfies~\eqref{eq:Psireg} and
$\chi\in L^\infty(Y,\R^{d+1})$. Then
\begin{itemize}
\item[(a)]
$a_n^{-1}\Psi_n\to_d X_\alpha$, where
$X_\alpha$ is as in~\eqref{eq:X} and~\eqref{eq:X2}.
\item[(b)]
Suppose further that $\Psi$ is nondegenerate. Then hypothesis~\ref{itm:N2} is satisfied.
\item[(c)]
Suppose moreover that $\Psi$ is nondegenerate and aperiodic, and that
$\chi|_A:A\to\R^{d+1}$ is uniformly continuous
for each $A\in\cA$.
Then hypothesis \ref{itm:N1} is satisfied with
$g_{d+1}$ being the density for the random vector $X_\alpha$
in part~(a).
\end{itemize}
\end{lemma}

\begin{proof}
Since $\chi$ is in $L^\infty$ and $\bar\pi$ is a measure-preserving semiconjugacy, it is immediate that $\bPsi$ is regularly varying function on $(\bY,\bmu)$ with the same index $-\alpha$  and spectral measure $\sigma$ as $\Psi$.
Moreover, $\bPsi$ is nondegenerate (resp.\ aperiodic) if and only if $\Psi$ is.

Since $\bPsi$ satisfies~\eqref{eq:Psireg}, 
Lemma~\ref{lem:GM}(a,b) applies to $\bPsi:\bY\to\R^{d+1}$ and hence to $\bPsi\circ\bar\pi:Y\to\R^{d+1}$. Since $\Psi_n=\bPsi_n\circ \bar\pi+\chi\circ F^n-\chi$,
it follows that parts~(a) and~(b) are satisfied for $\Psi:Y\to\R^{d+1}$.

Part~(c) follows by Theorem~\ref{thm:LLTtwo}.
\end{proof}
 
\begin{prop} \label{prop:N42}
Suppose that $(F,Y,\mu)$ is a two-sided Gibbs-Markov map and that $H:Y^\tau\to\R^d$ has the property that $H|_A:A\to\R^d$ is bounded and uniformly continuous on $A$ for each $A\in\cA$. Then $H$ satisfies hypothesis \ref{itm:N4}.
\end{prop}

\begin{proof}
This is almost identical to the proof of Proposition~\ref{prop:N4}.
\end{proof}

\begin{rmk} As in Section~\ref{sec:GM}, aperiodicity is typical and is a $C^2$ open $C^r$ dense condition. To see this, 
suppose that $\phi\circ F=\lambda e^{is\cdot \Psi}\phi$ with $\phi$ measurable.
Then $\tilde\phi\circ F=\lambda e^{is\cdot\bPsi\circ\bpi}\tilde\phi$
where $\tilde\phi=\phi\,e^{-is\cdot\chi}$.
Again, $\tilde\phi$ has a continuous version, and therefore so does $\phi$. Hence the existence of periodic orbits with good asymptotics for $\Psi$ rules out aperiodicity by the same argument as in Section~\ref{sec:GM}.
Provided $\Psi$ is $C^r$, $r\ge2$, on two partition elements of $Y$, good asymptotics holds for a $C^2$ and $C^r$ dense set of $\Psi$.
\end{rmk}

\section{Stable MLLT for Young tower maps}
\label{sec:YT}

The abstract stable MLLT for suspension flows in Section~\ref{sec:abstract} was implicitly formulated in the ``fully nonlattice situation'' where the joint observable $\Psi$ satisfies a nonlattice discrete time MLLT. This result was applied to $\R^d$-valued observables $w$  for flows under an aperiodicity assumption that is generically satisfied.

In the case of maps, the return time function (now denoted  by $R$ instead of $\tau$) is integer-valued. Consequently, the aperiodicity assumption in Definition~\ref{defn:aperiodic} automatically fails; similarly, when $w$ is $\Z^d$-valued.
In this section, we consider the main cases beyond the fully nonlattice case: fully lattice where $R$ is integer-valued and $w$ is $\Z^d$-valued, as well as hybrid situations where $R$ is integer-valued and $w$ is nonlattice.
(Similarly, we can deal with the hybrid case where $\tau$ is nonlattice and $w$ is lattice.)
We note that~\cite{DolgopyatNandori20} covers more general situations that are omitted here. 

In Subsection~\ref{sec:absYT}, we prove an abstract stable MLLT for maps where the observable $w$ is $\Z^d$-valued, and in Subsection~\ref{sec:hybrid} we consider the abstract hybrid case where $w$ is $\R^d$-valued. These two subsections represent the analogues for maps of the abstract flow results in Section~\ref{sec:abstract}. In Subsection~\ref{sec:GMmap}, we discuss how to verify hypotheses for Young tower maps (where the first return to the base is a Gibbs-Markov map);
this is the analogue of Sections~\ref{sec:GM} and~\ref{sec:GM2}.

\subsection{Abstract stable MLLT for tower maps and $\Z^d$-valued observables}
\label{sec:absYT}

In this subsection, we state an abstract fully lattice MLLT which is the analogue of the fully nonlattice MLLT stated in Theorem~\ref{thm:abstract}.

Let $(Y,\mu)$ be a probability space and let $F:Y\to Y$ be
a measure-preserving transformation.
Let $R:Y\to\Z^+$ be an integrable return time function and set $\bR=\int_Y R\,d\mu$.
Form the discrete suspension (tower)
\[
\Delta=\{(y,j)\in Y\times\Z:0\le j\le R(y)\}/\sim
\quad\text{where}\quad (y,R(y))\sim (Fy,0)\},
\]
with probability measure $\mu_\Delta=(\mu\times{\rm{counting}})/\bR$.
Define the measure-preserving map
$f_\Delta:\Delta\to \Delta$ given by
$f_\Delta(y,j)=(y,j+1)$ computed modulo identifications.

For $d\ge1$, let $w:\Delta\to\Z^d$ be integrable
with $\bw=\int_{\Delta}w\,d\mu_\Delta$.
Define 
\[
W:Y\to\Z^d, \qquad \Psi=(W,R):Y\to\Z^{d+1},
\]
where
\(
W(y)=\sum_{j=0}^{R(y)-1}w(y,j).
\)
Note that
$\int_Y\Psi\,d\mu=(\bw\bR,\bR)$.
We define $\hPsi^{(n)}=\Psi_n-\big[n\int_Y\Psi\,d\mu\big]$.

Let $\alpha\in(1,2]$, $\ell$, $\tilde\ell$, $a_n$ be as before. Fix a collection $\cA_0$ of measurable subsets of $Y$ and
 define $\cA=\{A_0\times\{j\}:A_0\in\cA_0,\,0\le j\le \inf_{A_0}R-1\}$.

Our main hypotheses become:
\begin{description}
\item[(L1)\label{itm:L1}] \textbf{Joint stable MLLT for $\Psi=(W,R)$.}
Let $A_0,B_0\in \cA_0$ and let $L>0$. Then 
\begin{align*}
\lim_{n\to\infty}
\sup_{\genfrac{}{}{0pt}{}{z\in\R^{d+1}}{|z|\le L}}
\Big|a_n^{d+1}\mu\big(y\in A_0\cap F^{-n}B_0:\, \hPsi^{(n)} & (y)= [za_n] \big)
\\[-2.75ex] &
- g_{d+1}(z)\mu(A_0)\mu(B_0)\Big|=0,
\end{align*}
where $g_{d+1}$ is a continuous density for a $(d+1)$-dimensional random vector.
\item[(L2)\label{itm:L2}] \textbf{Joint LLD for $\Psi=(W,R)$.}
There is a constant $C>0$ such that
\[
\mu\big(\hPsi^{(n)} =z\big)\le C \frac{n}{a_n^{d+1}}\, \frac{\tilde\ell(|z|)}{|z|^\alpha}
\quad\text{for all $n\ge1$, $z\in\Z^{d+1}$,}
\]
\item[(L3)\label{itm:L3}] \textbf{Extra assumption for $d\ge\alpha^2$.}
When $d\ge\alpha^2$, set $\gamma_0=(d-\alpha^2)/(d+1-2\alpha)\in[0,1)$.
We require that there exists $\gamma>\gamma_0$ such that
for any $q\in\R^d$, $k\in\Z$,
\begin{align*}
\sum_{1\le m\le n^\gamma}\mu\big(\Psi_m= ([qa_n]+[n\bw], n+k)\big)
=o(a_n^{-d}).
\end{align*}
\item[(L4)\label{itm:L4}] \textbf{Coboundary}
This hypothesis is unchanged from \ref{itm:N4} except for the simpler definition of $\cA$ and with
$H:\Delta\to\Z^d$ given by $H(y,j)=\sum_{k=0}^{j-1}w(y,k)$.
\end{description}

We can now state our stable MLLT for $w:\Delta\to\Z^d$.
Define the $d$-dimensional density
\begin{equation} \label{eq:gdL}
g(q)=\bR^{d/\alpha}g_d(q\bR^{1/\alpha}) 
\quad\text{where}\quad
g_d(q)=\int_{-\infty}^\infty g_{d+1}(q+\bw y,y)\,dy.
\end{equation}

 \begin{thm}[Stable MLLT for $w$]
 \label{thm:abstract-lattice}
Assume \ref{itm:L1}--\ref{itm:L4}.
 Let $A, B\in\cA$ and let $q\in\R^d$. Then 
\[
  \lim_{n\to\infty}a_n^d\mu_\Delta\Big(x\in A\cap f_\Delta^{-n}B:\, w_n(x)-[n\bw]= [qa_n]\Big)
  = g(q)\mu_\Delta(A)\mu_\Delta(B).
 \]
\end{thm}

\begin{proof}
This is similar to the proof of Theorem~\ref{thm:abstract}, so we only sketch the details. (Generally, the arguments are simpler, except that we have to take care of centring.)

Write $x=(y,j)\in \Delta$. Then
$f_\Delta^nx=(F^{N_n}y,j+n-R_{N_n}(y))$ where
$N_n=N_n(x)$ is such that
$R_{N_n}(y)\le j+n<R_{N_n+1}(y)$.
Note that
\[
w_n(x)=W_{N_n(x)}(y)+H(f_\Delta^nx)-H(x).
\]

By the argument in Subsection~\ref{sec:pf2}, it suffices to prove
the following analogue of Lemma~\ref{lem:abstract}.
Let $A=A_0\times\{i\},\,
B=B_0\times\{j\}\in\cA$ and let $q\in\R^d$. 
We must show that
\begin{equation} \label{eq:absYT}
\lim_{n\to\infty} a_n^d \mu_\Delta\Big(x\in A\cap f_\Delta^{-n}B: W_{N_n(x)}(y)=[qa_n]+[n\bw]\Big)=g(q)\mu_\Delta(A)\mu_\Delta(B).
\end{equation}

Define
\[
Q_m(n)=\mu \Big(y\in A_0\cap F^{-m}B_0:
\hPsi^{(m)}(y)=([qa_n]+[n\bw]-[m\bw\bR],n-[m\bR]+k)\Big),
\]
where $k=i-j$.
Arguing as before, for $L>0$,
\[
\bR \mu_\Delta\Big(x\in A\cap f_\Delta^{-n}B: W_{N_n(x)}(y)=[qa_n]+[n\bw]\Big)=S_1(n,L)+S_2(n,L)+o(a_n^{-d})
\]
where
\[
S_1(n,L)= \sum_{\stackrel{m>n^\gamma}{|m-n/\bR|\le La_n}}Q_m(n), \qquad
S_2(n,L)= \sum_{\stackrel{m>n^\gamma}{|m-n/\bR|> La_n}}Q_m(n).
\]
Here, $\gamma\in(\gamma_0,1)$ is as in \ref{itm:L3} for $d\ge\alpha^2$ and $\gamma=0$ otherwise.

For the range of $m$ in $S_1$, we use \ref{itm:L1} to show uniform convergence
\[
a_m^{d+1}Q_m(n)=
g_{d+1}(z_{m,n})\mu(A_0)\mu(B_0)+o(1) =\bR^2
g_{d+1}(z_{m,n})\mu_\Delta(A)\mu_\Delta(B)+o(1) ,
\]
where $z_{m,n}=\Big(\dfrac{qa_n+n\bw-m\bw\bR}{a_m},\dfrac{n-m\bR}{a_m}\Big)$.
Continuing as in the proof of~\eqref{eq:S1}, we obtain that
	\begin{align*}
		& a_n^d  S_1(n,L)
		\\ &\; =\bR^{2+(d+1)/\alpha} a_n^{-1}
\sum_{|j|\le La_n}
		g_{d+1}\big(q\bR^{1/\alpha}+\bw j\bR^{1+1/\alpha}a_n^{-1},j\bR^{1+1/\alpha} a_n^{-1}\big)\mu_\Delta(A)\mu_\Delta(B)+o(1). 
\end{align*}
	Approximating Riemann integrals by Riemann sums with increments of width $\bR^{1+1/\alpha} a_n^{-1}$ yields
\begin{align*}
	\bR^{1+1/\alpha} a_n^{-1}\sum_{|j|\le La_n}
	g_{d+1}\big(q\bR^{1/\alpha}+\bw & j\bR^{1+1/\alpha} a_n^{-1},j\bR^{1+1/\alpha} a_n^{-1}\big)
 \\ & \to  
	\int_{-L\bR^{1+1/\alpha}}^{L\bR^{1+1/\alpha}} g_{d+1}(q\bR^{1/\alpha}+\bw y,y)\,dy
\end{align*}
as $n\to\infty$.
Hence
\[
\lim_{L\to\infty}\lim_{n\to\infty} a_n^d S_1(n,L)=
 \bR^{1+d/\alpha} g_d(q\bR^{1/\alpha})\mu_\Delta(A)\mu_\Delta(B)=
 \bR g(q)\mu_\Delta(A)\mu_\Delta(B).
\]

For the range of $m$ in $S_2$, we have 
\[
|([qa_n]+[n\bw]-[m\bw\bR],n-[m\bR]+k)|\approx |n-m\bR|.
\]
Proceeding as in the proof of~\eqref{eq:S2}, we obtain
that
\[
\lim_{L\to\infty}\limsup_{n\to\infty} a_n^d S_2(n,L)=0.
\]
This completes the proof of~\eqref{eq:absYT}.
\end{proof}

\begin{rmk} \label{rmk:g}
In \ref{itm:N1}, $g_{d+1}$ is the limiting density corresponding to the observable
$\Psi=(W,\ttau)$ which is centred.
Hence, $g_d$ as defined in~\eqref{eq:gd} is the limiting density corresponding to the induced observable $W=\sum_{j=0}^{\tau-1}w\circ f_\Delta^j$.

The definition of $g_d$ in~\eqref{eq:gdL} takes into account the fact that $w$ and $W$ are no longer centred. Indeed, define the induced centred observable 
$Z=\sum_{j=0}^{R-1}(w-\bw)\circ f_\Delta^j$.
In \ref{itm:L1}, $g_{d+1}$ is the limiting density corresponding to the observable
$\Psi-\bPsi$ where $\Psi=(W,R)=(Z+\bw R,R)$,
In particular, $g_d$ is the limiting density corresponding to $Z$.

Hence we see that in both \ref{itm:N1} and \ref{itm:L1}, $g_d$ is the limiting density corresponding to the induced centred observable.
\end{rmk}

\begin{rmk} More generally, we can consider $W:Y\to\R^d$ cohomologous to a lattice-valued observable where the coboundary satisfies conditions like in \ref{itm:L4}.
\end{rmk}

\subsection{Abstract stable MLLT for tower maps and $\R^d$-valued observables}
\label{sec:hybrid}

There is one further case that we highlight in this paper:
namely, the hybrid nonlattice-lattice case where $R:Y\to\Z^+$ defines a tower map $f_\Delta:\Delta\to \Delta$ and
$w:\Delta\to\R^d$ is an integrable mean zero observable.
Accordingly, define
\[
\Psi=(W,R):Y\to\R^d\times\Z, \qquad
H:\Delta\to\R^d, 
\]
given by
\[
W(y)=\sum_{j=0}^{R(y)-1}w(y,j), \qquad
H(y,j)=\sum_{k=0}^{j-1}w(y,k).
\]
We define $\cA_0$ and $\cA$ as in Subsection~\ref{sec:absYT}.
The corresponding hypotheses (NL1)--(NL4) are obvious modifications of \ref{itm:N1}--\ref{itm:N4} and \ref{itm:L1}--\ref{itm:L4}. For example, (NL1) takes the form
\begin{description}
\item[(NL1)\label{itm:NL1}] \textbf{Joint stable MLLT for $\Psi=(W,R)$.}
Let $A_0,B_0\in \cA_0$. Let $K\subset\R^d$ be bounded with $|\partial K|=0$ and let $L>0$. Then 
\begin{align*}
\lim_{n\to\infty}
\sup_{\genfrac{}{}{0pt}{}{z\in\R^d,z'\in\R}{|z|+|z'|\le L}}
\Big|a_n^{d+1}  \mu\big(y\in A_0\cap F^{-n}B_0:\, 
  \hPsi^{(n)}(y&) \in (za_n,[z'a_n])+K\times\{0\} \big)
\\ & - g_{d+1}(z,z')\mu(A_0)\mu(B_0)|K|\Big|=0,
\end{align*}
where $g_{d+1}$ is a continuous density for a $(d+1)$-dimensional random vector.
\end{description}

 \begin{thm}[Stable MLLT for $w$]
 \label{thm:abstract-hybrid}
Assume (NL1)--(NL4).
 Let $A, B\in\cA$.
Let $K\subset\R^d$ be bounded with $|\partial K|=0$ and let $q\in\R^d$. 
Then 
\[
  \lim_{n\to\infty}a_n^d\mu_\Delta\Big(x\in A\cap f_\Delta^{-n}B:\, w_n(x)\in qa_n +K\Big)
  = g(q)\mu_\Delta(A)\mu_\Delta(B)|K|. 
\]
where $g(q)$ is defined as in~\eqref{eq:gdL} (with $\bw=0$).
\qed
\end{thm}

\subsection{Verification of \ref{itm:L1}--\ref{itm:L4} for Young tower maps}
\label{sec:GMmap}

Let $(F,Y,\mu)$ be a (possibly two-sided) mixing Gibbs-Markov map with partition $\beta$ and associated constant $\theta\in(0,1)$.
As in Subsections~\ref{sec:absYT} and~\ref{sec:hybrid},
we consider towers $(f_\Delta,\Delta,\mu_\Delta)$ 
with integrable return time $R:Y\to\Z^+$.
In addition, we assume that $R$ is constant on elements of $\beta$.
Then $f_\Delta$ is called a \emph{Young tower map}.
The situation for Young tower maps is similar to, but simpler than, that for the Gibbs-Markov (semi)flows considered in Sections~\ref{sec:GM} and~\ref{sec:GM2}.

The aperiodicity condition in Definition~\ref{defn:aperiodic} is modified to the condition
that $\phi\circ F=\lambda e^{is\cdot\Psi}\phi$ has no measurable solutions $\phi:Y\to\bbS^1$ for all 
$s\in S\setminus\{0\}$, 
$\lambda\in\bbS^1$.
Here, we take $S= [-\pi,\pi]^{d+1}$ in the lattice case $w:\Delta\to\Z^d$, and 
$S= \R^d\times[-\pi,\pi]$ in the hybrid case $w:\Delta\to\R^d$.

We have the following analogues of Lemma~\ref{lem:GM}
and Proposition~\ref{prop:N4}. We omit the proofs which go the same way as before. For brevity, we focus on the fully lattice case.

\begin{lemma} \label{lem:YT}
Let $(F,Y,\mu)$ be a mixing Gibbs-Markov map and let
$\Psi:Y\to\Z^{d+1}$ be an integrable observable that is constant on partition elements.

Suppose that
$\Psi$ is aperiodic, nondegenerate, regularly varying with index $-\alpha$,  $\alpha\in(1,2]$. 
Then
\begin{itemize}
\item[(a)]
$a_n^{-1}\big(\Psi_n-n\int_Y\Psi\,d\mu\big)\to_d X_\alpha$, where
$X_\alpha$ is as in~\eqref{eq:X} and~\eqref{eq:X2}.
\item[(b)]
Hypothesis \ref{itm:L2} is satisfied.
\item[(c)]
Define $\cA_0$ to be the collection of measurable subsets $A_0\subset Y$ with $\mu(\partial A_0)=0$.
Then hypothesis \ref{itm:L1} is satisfied with
$g_{d+1}$ being the density for the random vector $X_\alpha$
in part~(a).\qed
\end{itemize}
\end{lemma}

\begin{prop} \label{prop:L4}
Suppose that $(F,Y,\mu)$ is Gibbs-Markov and that $H:\Delta\to\Z^d$ has the property that $H|_A:A\to\Z^d$ is bounded and uniformly continuous on $A$ for each $A\in\cA$. Then $H$ satisfies hypothesis \ref{itm:L4}. \qed
\end{prop}

\begin{rmk} \label{rmk:relax}
The condition $d<\alpha^2$ can be relaxed as in Subsection~\ref{sec:relax}
with $S_t$ replaced by 
$S_n=\sum_{1\le m\le n^\gamma}\mu\big\{W_m=n\bw+O(a_n),\,R_m=n-k\big\}$ for $k$ fixed.
However, when $\bw{\neq0}$ in the fully lattice case, the assumption on $(\omega R)_n$ needs to be modified as follows.
 Define $Z=W-\omega R\in L^p$ and $\tZ=Z-\int_Y Z\,d\mu$.
 Then proceeding as before yields $\tZ_m=n\bw-(\omega R)_m+o(n)$ and we have to deal with the extra term $n\bw$. Assume either that there is a constant $c>|\bw|$ such that $|(\omega R)_m| \ge cR_m$ for all $m$, or alternatively that there is a constant $c\in(0,|\bw|)$ such that
$|(\omega R)_m|\le cR_m$ for all $m$. 
 Then \ref{itm:L3} holds provided $Z=W-\omega R\in L^p$ for some 
$p$ sufficiently large as before.
\end{rmk}

\begin{rmk} \label{rmk:aperiodic}
A sufficient condition for aperiodicity can be formulated as follows.
Suppose that $\phi:Y\to\bbS^1$ is measurable and satisfies $\phi\circ F=\lambda e^{is\cdot\Psi}\phi$ for some $s\in\R^{d+1}$ and $\lambda\in\bbS^1$.
By Liv\v{s}ic regularity~\cite[Theorem~1.1]{Gouezel06}, $\phi$ has a continuous version.
Let $y_1,y_2,\ldots\in \bigcup_{b\in\beta}b$ be fixed points for $F$.
Then $(\Psi(y_i)-\Psi(y_j))\cdot s\equiv0\bmod 2\pi$ for all $i,j$.
Form a $(d+1)\times(d+1)$ matrix $A\in \Z^{(d+1)\times (d+1)}$ with each row given by a vector of the form $\Psi(y_i)-\Psi(y_j)$ for some $i,j$. Then $As\equiv0\bmod (2\pi\Z)^{d+1}$. If $\det A=\pm1$, then inverting yields that $s\equiv0\bmod (2\pi\Z)^{d+1}$ and hence $\Psi$ is aperiodic.
\end{rmk}

\section{Stable MLLT for nonuniformly expanding maps and semiflows}
\label{sec:NUE}

In this section, we obtain stable MLLTs for nonuniformly expanding maps and semiflows. The maps and semiflows are defined respectively in Subsections~\ref{sec:NUEmap} and~\ref{sec:NUEflow}.
Stable MLLTs for such maps and semiflows are formulated in 
Subsection~\ref{sec:NUELLT}.

\subsection{Nonuniformly expanding maps}
\label{sec:NUEmap}

Let $T:M\to M$ be a measurable map defined on a bounded metric space
$(M,d)$.
Let $Y\subset M$ be a measurable subset supporting a Borel probability measure $\mu$. We suppose that $R:Y\to\Z^+$ is a measurable function satisfying $T^{R(y)}y\in Y$ for all $y\in Y$.
Define the induced map $F=T^R:Y\to Y$. 

We assume that 
\begin{itemize}
\item $(F,Y,\mu)$ is a mixing Gibbs-Markov map with partition $\beta$ and constant $\theta\in(0,1)$;
\item  $R$ is constant on elements of $\beta$;
\item  There is a constant $C>0$ such that $d(T^jy,T^jy')\le C d_\theta(y,y')$ for all $y,y'\in b$, $b\in\beta$, $0\le j\le R|_b$.
\end{itemize}

Using $F:Y\to Y$ and $R:Y\to\Z^+$, we define a Young tower
$(f_\Delta,\Delta,\mu_\Delta)$ as in Section~\ref{sec:GMmap}, and
the semiconjugacy
\[
\pi:\Delta\to M, \qquad \pi(y,j)=T^jy,
\]
between $f_\Delta$ and $T$.
Then $\mu_M=\pi_*\mu_\Delta$ is an ergodic
$T$-invariant probability measure on $M$. 
We say that
$(T,M,\mu_M)$ is a \emph{nonuniformly expanding map modelled by the Young tower $(f_\Delta,\Delta,\mu_\Delta)$}.

\subsection{Nonuniformly expanding flows}
\label{sec:NUEflow}

Now let $T_t:M\to M$ be a semiflow with bounded $C^\eta$ roof function
$r:X\to\R^+$ for some $\eta\in(0,1]$
and measurable Poincar\'e map $T=T_r:X\to X$ defined on some cross-section
$X\subset M$.
We suppose that $\inf_X r>0$ and 
that 
for any $K>0$ there exists $C>0$ such that
\[
d(T_tx,T_tx')\le Cd(x,x')
\quad\text{for $x,x'\in M$, $0\le t\le K$}.
\]

We require 
that the Poincar\'e map $T:X\to X$ is a nonuniformly expanding map
as in Section~\ref{sec:NUEmap}.
In particular, there is an induced mixing Gibbs-Markov map $(F,Y,\mu)$ where
$Y\subset X$,  $F=T^R:Y\to Y$ 
and $R:Y\to\Z^+$ is constant on partition elements.

Define the induced roof function
\[
\tau:Y\to\R^+, \qquad \tau=\sum_{j=0}^{R-1}r\circ T^j,
\]
and form the suspension flow $(F_t,Y^\tau,\mu^\tau)$ using
$(F,Y,\mu)$ and $\tau:Y\to\R^+$.
Then
\[
\pi:Y^\tau\to M, \qquad \pi(y,u)=T_uy,
\]
defines a semiconjugacy between $F_t$ and $T_t$, and 
$\mu_M=\pi_*\mu^\tau$ is an ergodic $T_t$-invariant probability measure on $M$.
We call $(T_t,M,\mu_M)$ a \emph{nonuniformly expanding semiflow modelled by the suspension semiflow $(F_t,Y^\tau,\mu^\tau)$}.

\subsection{Stable MLLT for nonuniformly expanding maps and flows}
\label{sec:NUELLT}

Let $(T_t,M,\mu_M)$ be a nonuniformly expanding semiflow as in Section~\ref{sec:NUEflow}.

Fix $\eta\in(0,1]$ and let $v:M\to\R^d$ be a bounded $C^\eta$ observable with $\int_M v \,d\mu_M=0$.
Define the lifted observable $w=v\circ\pi:Y^\tau\to\R^d$ and define
\[
\Psi=(W,\ttau):Y\to\R^{d+1}, \qquad
H:Y^\tau\to\R^d,
\] 
as in Section~\ref{subsec:abstract}.

\begin{lemma} \label{lem:NUEflow}
\begin{itemize}
\item[(a)] $\Psi$ satisfies~\eqref{eq:Psireg} with $\theta_1=\theta^\eta$ instead of $\theta$;
\item[(b)]  $H|_A$ is uniformly continuous for all 
subsets $A=A_0\times I\subset Y^\tau$ where $A_0\subset Y$ is measurable and
$I\subset [0,\inf_{A_0}\tau]$ is an interval.
\end{itemize}
\end{lemma}

\begin{proof}
(a) Define
\[
V:X\to\R^d, \qquad  V=\int_0^r v\circ T_u\,du.
\]
For $x,x'\in X$ with $r(x')\le r(x)$,
\begin{align*}
|V(x)-V(x')| & \le \int_0^{r(x)}|v(T_ux)-v(T_ux')|\,du \;+\;|v|_\infty |r(x)-r(x')|
\\
& \le |v|_\eta\int_0^{|r|_\infty}d(T_ux,T_ux')^\eta\,du \;+\;|v|_\infty |r|_\eta\,d(x,x')^\eta
\\
& \le C^\eta |v|_\eta |r|_\infty\,d(x,x')^\eta \;+\;|v|_\infty |r|_\eta\,d(x,x')^\eta
\ll d(x,x')^\eta.
\end{align*}
Hence $V$ is $C^\eta$.

Next, notice that
$W=\int_0^\tau v\circ T_u\,du=\sum_{j=0}^{R-1}V\circ T^j$. Hence for $y,y'\in b$, $b\in\beta$, 
\[
|W(y)-W(y')|\le |V|_\eta \sum_{j=0}^{R(y)-1}d(T^jy,T^jy')^\eta
\ll (R|_b) d_{\theta_1}(y,y').
\]
where $\theta_1=\theta^\eta$.
Replacing $V$ by $r$ yields
\[
|\ttau(y)-\ttau(y')|= |\tau(y)-\tau(y')|\ll (R|_b) d_{\theta_1}(y,y'),
\]
so $|\Psi(y)-\Psi(y')|\ll(R|_b) d_{\theta_1}(y,y')$.
But $\tau\ge(\inf r)R$ so $R|_b\ll (\inf r)^{-1} \infb\tau$.
Hence
\[
|\Psi(y)-\Psi(y')|\ll 
(\infb\tau) d_{\theta_1}(y,y')
=(\infb\ttau+\btau) d_{\theta_1}(y,y')
\ll (\infb|\Psi|+1) d_{\theta_1}(y,y'),
\]
verifying~\eqref{eq:Psireg}.

\vspace{1ex} \noindent (b) 
Note that $H(y,u)=\int_0^uv(T_ty)\,dt$.
Set $K=\sup I$.
Then $|H(y,u)|\le K|v|_\infty$ so $H|_A$ is bounded.

For $(y,u),\,(y',u)\in A$,
\begin{align*}
|H(y,u)-H(y',u)|
& \le K\sup_{t\in[0,K]}|v(T_ty)-v(T_ty')|
\\ & \le K|v|_\eta \sup_{t\in[0,K]}d(T_ty,T_ty')^\eta
\ll d(y,y')^\eta.
\end{align*}
Also, $|H(y,u)-H(y,u')|\le |v|_\infty\, |u-u'|$.
 Hence $H|_A$ is uniformly continuous.
\end{proof}

\begin{thm} \label{thm:NUEflow}
Let $(T_t,M,\mu_M)$ be a nonuniformly expanding semiflow.
Let $v:M\to\R^d$ be a bounded H\"older observable with $\int_M v\,d\mu_M=0$.

Suppose that
$\Psi$ is aperiodic, nondegenerate, regularly varying with index $-\alpha$,  $\alpha\in(1,2]$ and spectral measure $\sigma$. Let $g_{d+1}$ be the density for the corresponding stable law $X_\alpha$ in~\eqref{eq:X} and~\eqref{eq:X2} and
define $g$ as in~\eqref{eq:gd}.
 Assume that $d<\alpha^2$.

Let $\cA$ be the collection of measurable subsets $A\subset Y$ 
where
$\mu(\partial A)=0$.

Then $v$ satisfies a stable MLLT: setting $v_t=\int_0^t v\circ T_s\,ds$,
\[
  \lim_{t\to\infty}a_t^d\mu_M\Big(x\in A\cap T_t^{-1} B:\, v_t(x)\in qa_t+K\Big)
  = g(q)\mu_M(A)\mu_M(B)|K|
 \]
 for all $A, B\in\cA$, all
 $K\subset\R^d$ bounded with $|\partial K|=0$ and all $q\in\R^d$. 
\end{thm}

\begin{proof}
Since $\pi:Y^\tau\to M$ is a measure-preserving semiconjugacy, it suffices to prove the MLLT for $w=v\circ\pi:Y^\tau\to\R^d$.

By Theorem~\ref{thm:abstract}, we reduce to verifying hypotheses \ref{itm:N1}--\ref{itm:N4}
for $\Psi:Y\to\R^{d+1}$ and $H:Y^\tau\to\R^d$.

By Lemma~\ref{lem:NUEflow}(a), $\Psi$ satisfies~\eqref{eq:Psireg}, so \ref{itm:N1} and \ref{itm:N2} follow from Lemma~\ref{lem:GM}.
Hypothesis~\ref{itm:N4} follows from
Lemma~\ref{lem:NUEflow}(b) and Proposition~\ref{prop:N4}.
Hypothesis~\ref{itm:N3} is vacuously satisfied since $d<\alpha^2$.
\end{proof}

\begin{rmk}  \label{rmk:N3}
The assumption that $d<\alpha^2$ can be relaxed as in Section~\ref{sec:relax}.
\end{rmk}

\begin{rmk} \label{rmk:NUEmap}
The same result and proof holds (with the obvious modifications) for nonuniformly expanding maps $(T,M,\mu_M)$ and mean zero bounded H\"older observables $v:M\to\R^d$.
The conclusion is that
\[
  \lim_{n\to\infty}a_n^d\mu_M\Big(x\in A\cap T^{-n}B:\, v_n(x)\in qa_n +K\Big)
  = g(q)\mu_M(A)\mu_M(B)|K|. 
\]

Similarly, the result holds for 
nonuniformly expanding maps $(T,M,\mu_M)$ and H\"older observables $v:M\to\Z^d$,
 but H\"older means that $v$ is bounded and constant on $T^jb$ for all $b\in\beta$ and $0\le j<R|_b$.
The conclusion is that
\[
  \lim_{n\to\infty}a_n^d\mu_M\Big(x\in A\cap T^{-n}B:\, v_n(x)-[n\bv]= [qa_n]\Big)
  = g(q)\mu_M(A)\mu_M(B).
 \]
\end{rmk}

%
%
%

\section{Applications to intermittent dynamical systems}
\label{sec:PM}

In this section, we consider applications to intermittent dynamical systems.
We consider intermittent maps in Subsection~\ref{sec:LSV} and invertible intermittent solenoids in Subsection~\ref{sec:sol}. In Subsections~\ref{sec:intflow}, we consider intermittent semiflows and flows.

\subsection{Intermittent maps}
\label{sec:LSV}

We begin by considering a class of Pomeau-Manneville intermittent maps~\cite{PomeauManneville80}.
These are maps $T:M\to M$, where $M=[0,1]$, that are uniformly expanding except for the presence of neutral fixed points.  
To satisfy nondegeneracy of $\Psi$ in Section~\ref{sec:NUELLT} (recall Definition~\ref{defn:nondeg} and Proposition~\ref{prop:jointrv}), we require the existence of at least $d+1$ neutral fixed points. To fix ideas, we consider the map $(T,M,\mu_M)$ defined in~\eqref{eq:CFKM}.
(More general classes of Thaler maps~\cite{Thaler80} with two (or more) neutral fixed points are considered in~\cite[Section~5.3]{CMsub} and work the same way.)

The analysis of this example proceeds along familiar lines.
(For further details, see~\cite[Section~6.2]{CFKM20}.)
Set $Y=[\frac13,\frac23]$ and let $R:Y\to\Z^+$ be the first return time
$R(y)=\inf\{n{\ge1}:T^ny\in Y\}$. 
Then the first return map $F=T^R:Y\to Y$ is Gibbs-Markov with partition $\beta$ given by maximal intervals on which the return time is constant. There exists a unique absolutely continuous invariant probability measure $\mu$ on $Y$ and moreover $F$ is mixing. 
Since $F'\ge3$, it is easily  verified that $|y-y'|\ll d_\theta(y,y')$ for all $\theta\in[\tfrac13,1)$.
Since $T'\ge1$, it is immediate that $|T^jy-T^jy'|\le |Fy-Fy'|\ll d_\theta(y,y')$ for
$y,y'\in b$, $b\in\beta$, $0\le j\le R|_b$.
Hence as discussed in Section~\ref{sec:NUEmap}, 
$(T,M,\mu_M)$ is a nonuniformly expanding map modelled by the Young tower $(f_\Delta,\Delta,\mu_\Delta)$ defined in Section~\ref{sec:absYT}.

The sets $Y_0=(\frac13,\frac12)$, $Y_1=(\frac12,\frac23)$ are unions of partition elements and
there is a constant $c>0$ such that $\mu(R1_{Y_i}>t)\sim ct^{-\alpha}$ as $t\to\infty$ for $i=0,1$.

Now let $v:M\to\R$ be an $\eta$-H\"older observable, $\eta\in(0,1]$,
with $\int_M v\,d\mu_M=0$.
(We focus on the hybrid case where the dynamics is discrete but observables are nonlattice. The fully lattice case where $v$ takes values in $\Z$ and is constant on a neighbourhood of each neutral fixed point and on $T^j b$ for $b\in\beta$ and $0\le j<R|_b$, is simpler.)
Define the induced observable
\(
W:Y\to\R$, $W=\sum_{j=0}^{R-1}v\circ T^j.
\)

\begin{prop} \label{prop:LSV}
Define $\omega:Y\to\R$, $\omega|_{Y_i}=v(i)$, $i=0,1$.
Then $W=\omega R+Z$ where $Z\in L^p$ for some $p>\alpha$.
Moreover, $Z\in L^\infty$ provided that $\eta\in(1/\alpha,1]$.
\end{prop}

\begin{proof}
Shrinking $\eta$ if necessary, we can suppose that $\eta\neq 1/\alpha$.
We recall the standard calculation (see for example~\cite{Gouezel04}).
For $y\in Y_0$,  
\[
|Z(y)|=|W(y)-v(0)R(y)|
=\Big|\sum_{j=0}^{R(y)-1}(v(T^j y)-v(0))\Big|
\le |v|_\eta \sum_{j=0}^{R(y)-1}(T^j y)^\eta.
\]
We have $T^jy\approx (R(y)-j)^{-\alpha}$, so
\[
|Z(y)|\ll \sum_{j=1}^{R(y)}j^{-\eta\alpha}\ll 
  1+ R(y)^{1-\eta\alpha}. 
\]
Similarly for $y\in Y_1$, and the result follows.
\end{proof}

\begin{pfof}{Theorem~\ref{thm:CFKM}}
By Proposition~\ref{prop:LSV}, we can apply Proposition~\ref{prop:jointrv}
to deduce that $\Psi=(W,R)$ is regularly varying and nondegenerate. 
By assumption, $d=1<\alpha^2$ and $\Psi$ is aperiodic.
Hence we have verified the assumptions of Theorem~\ref{thm:NUEflow}
(more precisely, Remark~\ref{rmk:NUEmap}).
In this way, we obtain a stable MLLT relative to intervals $A$, 
$B\subset Y$.

For intermittent maps, it is well-known how to enlarge $Y$ to contain $[\eps,1-\eps]$ for $\eps>0$ arbitrarily small. Hence we can take 
$A$ and $B$ to be arbitrary intervals in $M$ that are bounded away from $0$ and $1$. The result is now a consequence of Lemma~\ref{lem:vague}.~
\end{pfof}

\paragraph{Thaler maps}
More generally, we can consider intermittent maps 
with $k$ neutral fixed points $x_i$, $i=1,\dots,k$, all with same exponent
$1+1/\alpha$ where $\alpha\in(1,2]$, and observables $v$ with values in $\R^d$.
For examples where there are $k$ full branches each containing one neutral fixed point, see~\cite[Section~7]{Sera20},~\cite[Section~4.1]{CMT26} or~\cite[Example~5.12]{CMsub}. (The range of $\alpha$ in these papers is different from here, but that does not affect the arguments.)

The treatment is almost identical to that for~\eqref{eq:CFKM} with obvious modifications when $d\ge2$. In particular, there is a mixing Gibbs-Markov first return map $F=T^R$ and we have a decomposition $W=\sum_{i=1}^k \omega_iR_i+Z$ where $\omega_i=v(x_i)$ and each $R_i$ is regularly varying with index $-\alpha$. Moreover, $Z$ satisfies Proposition~\ref{prop:LSV}.

To obtain nondegeneracy, we require that the vectors $\tilde\omega_i=(\omega_i,1)$ span $\R^{d+1}$, so in particular $k\ge d+1$. Also, if $d\ge\alpha^2$ then we can invoke Proposition~\ref{prop:relax} to ensure that (NL3) is satisfied.
It suffices that (i) each $\omega_i$ has positive component in say the first coordinate, so~\eqref{eq:relax} is satisfied; and (ii) $\eta$ is sufficiently close to $1$ so, by Proposition~\ref{prop:LSV}, the constraint on $p$ in Proposition~\ref{prop:relax} is satisfied.

\subsection{Intermittent solenoids}
\label{sec:sol}

A class of nonuniformly hyperbolic diffeomorphisms can be constructed
as in~\cite{AlvesAzevedo16,AlvesPinheiro08,MV16}.
 along the lines of Smale-Williams solenoids~\cite{Smale67,Williams67}.

For example, consider a Thaler map $\bT:[0,1]\to[0,1]$ with two full orientation-preserving branches 
on $[0,\frac12]$ and $[\frac12,1]$ and neutral fixed points at $0$ and $1$. 
Modify $\bT$ in a neighbourhood of $\frac12$ so that $\bT$ is $C^2$ on $(0,1)$ working~$\bmod\, 1$.
Define $M=[0,1]\times\D$ where $\D$ is a $(\dim M-1)$-dimensional disk
and define $T_0:M\to M$ by $T_0(x_1,x_2)=(\bT x_1,\tfrac12 x_2)$.
Perturb to obtain an embedding $T:M\to M$ of the form 
$T(x_1,x_2)=(\bT x_1,g(x_1,x_2))$ where $g:M\to\D$ is $C^2$ such that 
$|\partial_{x_2}g(x_1,x_2)|\le\frac34$.

Let $\bY\subset [0,1]$ be a suitable inducing set for $\bT$ with first return mixing Gibbs-Markov map $\bF=\bT^R:\bY\to\bY$.  Let
$Y=\bY\times[0,1]$. Then
$F=T^R:Y\to Y$ is a two-sided Gibbs-Markov map 
and the measure-preserving semiconjugacy $\bpi:Y\to\bY$ is simply projection onto the first coordinate.

Following Section~\ref{sec:YT}, we can use $(F,Y,\mu)$ and the first return time $R$ to build a two-sided Young tower map $(f_\Delta,\Delta,\mu_\Delta)$ which models a nonuniformly hyperbolic map $(T,M,\mu_M)$.

Let $v:M\to\R^d$ be H\"older with $\int_M v\,d\mu_M=0$ and define the lifted observable
$w=v\circ \pi:\Delta\to\R^d$.
Define $\Psi=(W,R):Y\to\R^d\times\Z$ as in Section~\ref{sec:YT}. 
For typical $v$, we obtain that $\Psi$ is regularly varying, nondegenerate, and aperiodic.

By Section~\ref{sec:hybrid}, to prove a stable MLLT for $v$, it suffices to verify hypotheses~\ref{itm:NL1}--(NL4).
Here the exponential contraction along stable leaves is crucial.
By the standard Sinai argument, we can write
$\Psi=\bPsi\circ\bpi+\chi\circ F-\chi$ where $\bPsi:\bY\to\R^d\times\Z$ satisfies~\eqref{eq:Psireg} and $\chi:Y\to\R^d\times\Z$ is uniformly continuous.
Hypotheses~\ref{itm:NL1}, (NL2) and (NL4) follow readily.
As usual, hypothesis~(NL3) is vacuously satisfied if $d<\alpha^2$ and is also satisfied if for instance $v(0)$ and $v(1)$ both have positive first component and $\eta$ is sufficiently close to $1$.

The construction can be modified so that stable leaves are not exponentially contracting (see for instance~\cite[Example~4.2]{MV16}).
If we are in the fully lattice case $v:M\to\Z^d$, then $\Psi$ is automatically constant on stable leaves and we can take $\chi=0$. More generally, it suffices that $v\circ\pi=\hat w+\chi\circ f_\Delta-\chi$ where $\hat w$ takes values in $\Z^d$ and $\chi:\Delta\to\R^d$ is uniformly continuous.

However, without assuming exponential contraction along stable leaves, we do not know how to treat general H\"older observables.
The situation was discussed in detail in~\cite[Section~8]{BBM19} and similar considerations apply here.

\subsection{Intermittent semiflows and flows}
\label{sec:intflow}

Now let $(T_t,M,\mu_M)$ be a nonuniformly expanding semiflow
(Section~\ref{sec:NUEflow}) such that the Poincar\'e map $T=T_r:X\to X$ is an intermittent map as in Subsection~\ref{sec:LSV}.
The neutral fixed points $x_i$ for $T$ become neutral periodic orbits $p_i$
of period $r(x_i)$ for the semiflow $T_t$.
For definiteness, we take $T$ to be the map in~\eqref{eq:CFKM}.

Given a $C^\eta$ observable $v:M\to\R$ with $\int_M v\,d\mu_M=0$, we define the lifted observable $w=v\circ \pi:Y^R\to\R$ and the induced observable
$W:Y\to\R^2$
as in Section~\ref{sec:abstract}.

Let $x_0,\,x_1$ denote the neutral fixed points for $T$ and define
\[
\omega_i=r(x_i)^{-1}\int_0^{r(x_i)}v(T_ux_i)\,du.
\]
We have the following counterpart to Proposition~\ref{prop:LSV}.

\begin{prop} \label{prop:LSV1}
Define $\omega:Y\to\R$, $\omega|_{Y_i}=\omega_i$, $i=0,1$.
Then $W=\omega\tau+Z$ where $Z\in L^p$ for some $p>\alpha$.
Moreover, $Z\in L^\infty$ provided that $\eta\in(1/\alpha,1]$.
\end{prop}

\begin{proof}
Note that
\[
W(y)=\int_0^{\tau(y)} v(y,u)\,du=
\sum_{j=0}^{R(y)-1}\int_0^{r(T^jy)}v(T_u T^jy)\,du.
\]
In particular,
\[
W(x_i)=
\sum_{j=0}^{R(x_i)-1}\int_0^{r(x_i)}v(T_ux_i)\,du
=R(x_i)r(x_i)\omega_i = \omega_i\tau(y).
\]
Therefore for $y\in Y_i$,  
\[
Z(y)
=W(y)-W(x_i)=\sum_{j=0}^{R(y)-1}\Big\{\int_0^{r(T^j y)}v(T_uT^j y)\,du-
\int_0^{r(x_i)}v(T_ux_i)\,du\Big\}.
\]
Now $\int_0^{r(T^j y)}v(T_uT^j y)\,du- \int_0^{r(x_i)}v(T_ux_i)\,du=I_1+I_2$ where
\[
I_1  =\int_0^{r(T^j y)}\big\{v(T_uT^j y)-
v(T_ux_i)\big\}\,du, \qquad
 I_2  = \int_{r(x_i)}^{r(T^j y)}v(T_uT^j y)\,du.
\]

Our assumptions on $v$, $r$ and $T_t$ ensure that
\[
|I_1|,\,|I_2| \ll |T^j y-x_i|^\eta
\approx |R(y)-j|^{-\alpha\eta}.
\]
Hence
\[
|Z(y)| 
\ll
\sum_{j=1}^{R(y)}j^{-\alpha\eta}
\approx R(y)^{1-\alpha\eta}+1.
\]
(We suppose without loss that $\eta\neq 1/\alpha$.)
The result follows.
\end{proof}

\begin{cor} \label{cor:LSVflow}
Let $T_t:M\to M$ be the intermittent semiflow above with parameter $\alpha\in(1,2]$. Suppose that $v:M\to\R$ is $C^\eta$ and that
$\omega_i$ are nonzero and not all equal.
Assume that $\Psi=(W,\ttau):Y\to\R^2$ is aperiodic.
Then the stable MLLT holds:
\[
\lim_{t\to\infty}a_t\int_M \zeta_1\,\zeta_2\circ T_t\,\xi\circ(v_t-qa_t)\,d\mu_M
=g(q)\int_M \zeta_1\,d\mu_M \int_M \zeta_2\,d\mu_M \int_{\R^d} \xi\,d\Leb,
\]
for all $\zeta_1,\,\zeta_2\in C_c\big(M\setminus\{p_0,p_1\}\big)$ and $\xi\in C_c(\R)$.
\end{cor}

\begin{proof}
By Proposition~\ref{prop:LSV1}, we can apply Proposition~\ref{prop:jointrv} to deduce that $\Psi$ is regularly varying and nondegenerate.
By assumption, $d=1<\alpha^2$ and $\Psi$ is aperiodic.
Hence we have verified the assumptions of Theorem~\ref{thm:NUEflow} 
and the result follows.
\end{proof}

\begin{rmk} 
Intermittent semiflows with several neutral periodic orbits and $\R^d$-valued observables work the same way, subject to the restrictions concerning $d$.
Our results also cover invertible intermittent flows, but we require either exponential contraction along stable leaves or restrictions on the observable.
\end{rmk}

\appendix
\renewcommand{\thesubsection}{\Alph{section}.\arabic{subsection}}

\section{Vague convergence lemma}
\label{app:vague}

Throughout this appendix, all measures are non-negative and all measures on $\R^m$ are Borel measures.
Let $\Leb$ denote Lebesgue measure on $\R^m$ and
let $\cE$ denote the collection of bounded measurable subsets $E\subset\R^m$ with $\Leb(\partial E)=0$.
Let $C_c(\R^m)$ denote the space of continuous compactly supported functions $g:\R^m\to\R$.

Now, let $(\Omega,\PP)$ be a probability space and let
$\cA$ be some collection of measurable subsets of $\Omega$.
Let $F_t:\Omega\to\Omega$, $t>0$, be a family of measurable transformations and let
$\zeta_{t,\gamma}:\Omega\to\R^m$, $t>0$, $\gamma\in \Gamma$, be a parameterised family of measurable functions.
Also, let $a_\gamma:(0,\infty)\to(0,\infty)$, $\gamma\in\Gamma$, be a parameterised function.

We consider measurable functions
$H:\Omega\to\R^m$ satisfying the property
\begin{description}
\item[(H)\label{itm:H}]
For any $A\in\cA$ and any $\delta>0$, there exist
disjoint subsets $A_1,\dots,A_n\subset A$ with $A_1,\dots,A_n\in\cA$ such that
\[
\PP\Big(A\setminus \bigcup_{j=1}^n A_j\Big) < \delta
\quad\text{and} \quad
\max_{1\le j\le n} \diam H(A_j)< \delta .
\]
In addition, $H(A)$ is a bounded subset of $\R^m$ for each $A\in\cA$.
\end{description}

\begin{lemma} \label{lem:vague}
Suppose that
\begin{equation} \label{eq:vague}
\lim_{t\to\infty}a_\gamma(t) \PP (\omega\in A\cap F_t^{-1}B:\,
\zeta_{t,\gamma}(\omega)\in E) = \PP(A)\PP(B)\Leb(E)
\end{equation}
uniformly in $\gamma\in\Gamma$,
for all $A,B\in\cA$, $E\in\cE$.
Suppose that $H_1,H_2:\Omega\to\R^m$ satisfy property~\ref{itm:H}.
Then
\[
\lim_{t\to\infty}a_\gamma(t) \PP \left(\omega\in A\cap F_t^{-1}B:\, \zeta_{t,\gamma}(\omega)+H_1(\omega)+H_2(F_t\omega)\in E\right) = \PP(A)\PP(B)\Leb(E)
\]
uniformly in $\gamma\in\Gamma$,
for all $A,B\in\cA$, $E\in\cE$.
\end{lemma}

\begin{proof}
Fix $A,\,B\in\cA$ and 
introduce the family of measures
$\Q_t(D)=a_\gamma(t)\, \PP(A\cap F_t^{-1}B\cap D)$ on $\R^m$.
Also, let $\Q(D)=\PP(A)\PP(B)\Leb(D)$.
Define
\[
S^{t,\gamma}:\Omega\to\R^m, \qquad
S^{t,\gamma}=\zeta_{t,\gamma}+H_1+H_2\circ F_t.
\]
We are required to prove that
$\lim_{t\to\infty}S^{t,\gamma}_*\,\Q_t(E)=\Q(E)$ for all
$E\in\cE$.
Note that $S^{t,\gamma}_*\,\Q_t(\R^m)=a_\gamma(t)$, so the measures $S^{t,\gamma}_*\,\Q_t$ are trivially Radon measures (finite on compact sets).
Also $\Q$ is a Radon measure.
By Lemma~\ref{lem:Port} below,
 it suffices to prove that
\begin{equation} \label{eq:port}
\lim_{t\to\infty}\int_{\R^m}g\,d S^{t,\gamma}_*\,\Q_t=\int_{\R^m}g\,d\Q
\end{equation}
uniformly in $\gamma$, for all $g\in C_c(\R^m)$.

Let $g\in C_c(\R^m)$, and define
\[
G:(\R^m)^3\to \R,
\qquad G(x,y,z)=g(x+y+z).
\]
We claim that
\begin{align} \label{eq:vagueclaim} \nonumber
\lim_{t\to +\infty}a_\gamma(t)&\int_{A\cap F_t^{-1}B}G(H_1(\omega),H_2(F_t \omega),\zeta_{t,\gamma}(\omega))\, d\PP(\omega)\\
&=\int_{A\times B\times\R^m}G(H_1(\omega),H_2(\omega'),z)d(\PP\times\PP\times\Leb)(\omega,\omega',z)  
\end{align}
uniformly in $\gamma$.
Suppose that the claim is true. Then by definition of $G$,
\begin{align*}
\lim_{t\to +\infty}a_\gamma(t)&\int_{A\cap F_t^{-1}B}g(\zeta_{t,\gamma}(\omega)+H_1(\omega)+H_2(F_t \omega))\, d\PP(\omega)\\
&=\int_{A\times B\times\R^m}g(H_1(\omega)+H_2(\omega')+z)d(\PP\times\PP\times\Leb)(\omega,\omega',z) 
\\ &=\int_{A\times B\times\R^m}g(z)d(\PP\times\PP\times\Leb)(\omega,\omega',z) 
\end{align*}
uniformly in $\gamma$,
yielding~\eqref{eq:port}.

It remains to prove the claim.
	Since $g$ is compactly supported and $H_1(A)$, $H_2(B)$ are bounded, we can find a closed disk $D\subset\R^m$ 
	such that $\supp G\subset \R^m\times\R^m\times D$.
For any $\eps>0$,
there exists $\delta\in(0,\eps/\{\Leb(D)(\| G\|_\infty+1)\})$ such that 
\begin{equation} \label{eq:G}
|G(x,y,z)-G(x',y',z')|<\eps \quad\text{for all}\quad
|x-x'|<\delta,\ |y-y'|<\delta,\ |z-z'|<\delta.
\end{equation} 
We partition $D$ into a finite number of disjoint compact
subsets $C_1,C_2,\dots$ with $\Leb(\partial C_k)=0$ and $\diam C_k<\delta$.

Since $H_1$, $H_2$ satisfy \ref{itm:H}, we can choose
disjoint subsets $A_i\subset A$ and disjoint subsets 
$B_j\subset B$ with $A_1,\dots,A_m;\,B_1,\dots,B_n\in\cA$ such that
\begin{alignat}{2} \label{eq:N4a}
& \PP(A\setminus A') < \delta, & \qquad 
& \PP(B\setminus B') < \delta, \\
\label{eq:N4b}
& \max_{1\le i\le m} \diam H_1(A_i)  < \delta , & \qquad
& \max_{1\le j\le n} \diam H_2(B_j)  < \delta ,
\end{alignat}
where $A'=\bigcup_{i=1}^mA_i$ and $B'=\bigcup_{j=1}^nB_j$.

Now,
\begin{align*}
a_\gamma(t) & \int_{A\cap F_t^{-1}B}  G(H_1(\omega),H_2(F_t \omega),\zeta_{t,\gamma}(\omega))\, d\PP(\omega)
\\ & =
a_\gamma(t)\int_{A\cap F_t^{-1}B\cap\zeta_{t,\gamma}^{-1}D}  G(H_1(\omega),H_2(F_t \omega),\zeta_{t,\gamma}(\omega))\, d\PP(\omega)
\\ & =
a_\gamma(t)\int_{A'\cap F_t^{-1}B'\cap\zeta_{t,\gamma}^{-1}D}G(H_1(\omega),H_2(F_t \omega),\zeta_{t,\gamma}(\omega))\, d\PP(\omega)+I_1(t)+I_2(t),
\end{align*}
where
\begin{align*}
|I_1(t,\gamma)| & \le \|G\|_\infty\, a_\gamma(t)\,
\PP\big((A\setminus A')\cap F_t^{-1}B\cap \zeta_{t,\gamma}^{-1}D\big), \\
|I_2(t,\gamma)| & \le \|G\|_\infty\, a_\gamma(t)\,
\PP\big(A'\cap F_t^{-1}(B\setminus B')\cap \zeta_{t,\gamma}^{-1}D\big).
\end{align*}
By~\eqref{eq:vague} and~\eqref{eq:N4a},
\[
\lim_{t\to\infty}a_\gamma(t)\,
\PP\big((A\setminus A')\cap F_t^{-1}B\cap \zeta_{t,\gamma}^{-1}D\big)
=\PP(A\setminus A')\PP(B)\Leb(D)<\delta\Leb(D)
\]
uniformly in $\gamma$.
Hence $\limsup_{t\to\infty}\sup_{\gamma\in\Gamma}|I_1(t,\gamma)|<\delta\|G\|_\infty\Leb(D)<\eps$.
Similarly, $\limsup_{t\to\infty}\sup_{\gamma\in\Gamma}|I_2(t,\gamma)|<\eps$.

Next, by~\eqref{eq:G} and~\eqref{eq:N4b},
\begin{align*}
a_\gamma(t) & \int_{A'\cap F_t^{-1}B'\cap\zeta^{-1}D} G(H_1(\omega),H_2(F_t \omega),\zeta_{t,\gamma}(\omega))\, d\PP(\omega) \\
&= \sum_{i,j,k} a_\gamma(t) \int_{A_i\cap F_t^{-1}B_j\cap \zeta_{t,\gamma}^{-1}C_k}G(H_1(\omega),H_2(F_t \omega),\zeta_{t,\gamma}(\omega))\, d\PP(\omega) \\
&= \sum_{i,j,k}  \Big\{\min_{H_1(A_i)\times H_2(B_j)\times C_k}G\Big\} a_\gamma(t)\PP(A_i\cap F_t^{-1}B_j\cap \zeta_{t,\gamma}^{-1}C_k)+I_3(t,\gamma),
\end{align*}
where
\[
|I_3(t,\gamma)|\le \delta\sum_{i,j,k} a_\gamma(t)\PP(A_i\cap F_t^{-1}B_j\cap \zeta_{t,\gamma}^{-1}C_k).
\]
Hence $\limsup_{t\to\infty}|I_3(t,\gamma)|\le \delta\PP(A)\PP(B)\Leb(D)<\eps$.

By~\eqref{eq:vague},
\begin{align*}
 \lim_{t\to\infty} \sum_{i,j,k} &   \Big\{\min_{H_1(A_i)\times H_2(B_j)\times C_k}G\Big\} a_\gamma(t)\PP(A_i\cap F_t^{-1}B_j\cap \zeta_{t,\gamma}^{-1}C_k)
\\ & \qquad =
 \sum_{i,j,k}    \Big\{\min_{H_1(A_i)\times H_2(B_j)\times C_k}G\Big\} \PP(A_i)\PP(B_j)\Leb(C_k).
\end{align*}
Reversing the first two steps of the calculations in a simpler situation,
\begin{align*}
 \sum_{i,j,k}    & \Big\{\min_{H_1(A_i)\times H_2(B_j)\times C_k}G\Big\} \PP(A_i)\PP(B_j)\Leb(C_k).
\\ & =
\int_{A\times B\times \R^m}G(H_1(\omega),H_2(\omega'),z)\,d(\PP\times\PP\times\Leb)(\omega,\omega',z)+I_4,
\end{align*}
where $|I_4|\le \delta\Leb(D)+2\delta\|G\|_\infty\Leb(D)<3\eps$.

Since $\eps>0$ is arbitrary, the claim~\eqref{eq:vagueclaim} follows.
\end{proof}

A Portmanteau theorem for vague convergence can be found for example 
in~\cite[Theorem~3.2]{Resnick07}.
In the proof of Lemma~\ref{lem:vague}, we required a uniform version. Since we could not find such a result in the literature, we state and prove this here.

\begin{lemma} \label{lem:Port}
Let $\lambda_t^\gamma$, $t>0$, $\gamma\in\Gamma$, and $\lambda$ be Radon measures on $\R^m$. 
Let $\cE_\lambda$ denote the collection of bounded measurable subsets of $\R^m$ with $\lambda(\partial E)=0$.
Then the following are equivalent:
\begin{itemize}
\item[(a)] $\int_{\R^m} g\,d\lambda_t^\gamma\to \int_{\R^m}g\,d\lambda$
uniformly in $\gamma$ for each $g\in C_c(\R^m)$.
\item[(b)] $\lim_{t\to\infty}\lambda_t^\gamma(E)= \lambda(E)$ uniformly in $\gamma$
for each $E\in\cE_\lambda$.
\end{itemize}
\end{lemma}

\begin{proof}
First, suppose that (a) holds and let $K\subset \R^m$ be compact.
Choose a sequence of compact sets $K_j\subset\R^m$ and a sequence
of functions
$f_j\in C_c(\R^m)$ such that
$K_{j+1}\subset K_j$, $\bigcap_j K_j=K$, and
$1_K\le f_j\le 1_{K_j}$ for all $j$.
Then by (a), for each fixed $j$,
\[
\lambda_t^\gamma(K)\le \int_{\R^m}f_j\,d\lambda_t^\gamma\to \int_{\R^m}f_j\,d\lambda
\]
uniformly in $\gamma$.
Hence $\limsup_{t\to\infty} \sup_\gamma \lambda_t^\gamma(K)\le \int_{\R^m}f_j\,d\lambda$.
Also, $|f_j|\le 1_{K_1}$ which is $\lambda$-integrable and $f_j\to 1_K$ pointwise as $j\to\infty$, so by the dominated convergence theorem $\limsup_{t\to\infty} \sup_\gamma \lambda_t^\gamma(K)\le \lambda(K)$.

Second, continue to suppose that (a) holds and let $G\subset\R^m$ be open and bounded. 
Choose a sequence of open sets $G_j\subset\R^m$ and a sequence of functions
$f'_j\in C_c(\R^m)$ such that
$G_j\subset G_{j+1}$, $\bigcup_j G_j=G$, and
$1_G\ge f'_j\ge 1_{G_j}$. Then the same argument as before (but with dominating function $1_G$) shows that
$\liminf_{t\to\infty} \inf_\gamma \lambda_t^\gamma(G)\ge \lambda(G)$.

Now let $E\in\cE_\lambda$. Taking $K=\bar E$ and noting that $E\subset K$ and $\lambda(K)=\lambda(E)$, it follows from the first calculation that
$\limsup_{t\to\infty} \sup_\gamma \lambda_t^\gamma(E)\le \lambda(E)$.
Similarly, taking $G=\Int E$ it follows from the second calculation that
$\liminf_{t\to\infty} \inf_\gamma \lambda_t^\gamma(E)\ge \lambda(E)$.
This completes the proof that (a) implies (b).

Next, we show that (b) implies (a). Let $g\in C_c(\R^m)$, and set $M=\sup|g|$. 
There exists $R>0$ such that $\supp g\subset B_R(0)$.
Note that $\lambda(\partial B_r(0))>0$ for at most countably many $r\in [R, R+1]$. Hence there exists $r>0$ such that $D=\overline{B_r(0)}$ contains $\supp g$ and $D\in\cE_\lambda$.

Given $\eps>0$, choose $c_0<c_1<\dots<c_k$ such that $[-M,M]\subset(c_0,c_k)$ and $c_i-c_{i-1}<\eps$ for $i=1,\dots,k$.
Since $\lambda(g=c)>0$ for at most countably many values of $c$, we can arrange that $\lambda(g=c_i)=0$ for $i=1,\dots,k$.

Let $E_i=g^{-1}(c_{i-1},c_i)$. Then $\partial E_i\subset \{g=c_{i-1}\}\cup\{g=c_i\}$ so $\lambda(\partial E_i)=0$.
Hence $E_i\cap D\in\cE_\lambda$.
(Note that one of the $E_i$ contains $\R^m\setminus\supp g$ which is unbounded, hence the need to intersect with the bounded set $D$.)
By (b), there exists $T>0$ such that
\begin{equation} \label{eq:D1}
\Big|\sum_{i=1}^k c_i\lambda_t^\gamma(E_i\cap D)-
\sum_{i=1}^k c_i\lambda(E_i\cap D)\Big|\le \eps ,
\end{equation}
for all $t>T$, $\gamma$.

By definition of $E_i$, 
\[
c_{i-1}\lambda_t^\gamma(E_i\cap D)\le 
\int_{E_i\cap D}g\,d\lambda_t^\gamma
\le c_i\lambda_t^\gamma(E_i\cap D),
\]
for all $t$, $\gamma$, $i$. Since $\int_{\R^m}g\,d\lambda_t^\gamma=\int_Dg\,d\lambda_t^\gamma$, it follows that
\begin{equation} \label{eq:D2}
\Big|\int_{\R^m} g\,d\lambda_t^\gamma-\sum_{i=1}^k c_i\lambda_t^\gamma(E_i\cap D)\Big|<\eps \lambda_t^\gamma( D),
\end{equation}
for all $t$, $\gamma$.
Similarly,
\begin{equation} \label{eq:D3}
\Big|\int_{\R^m} g\,d\lambda-\sum_{i=1}^k c_i\lambda(E_i\cap D)\Big|<\eps \lambda(D).
\end{equation}
Since $D\in\cE_\lambda$, by (b) we can increase $T$ if need be so that $\lambda_t^{\gamma}(D)\le \lambda(D)+1$
for all $t>T$, $\gamma$.
This together with~\eqref{eq:D1},~\eqref{eq:D2},~\eqref{eq:D3}, yields that
\[
\sup_\gamma\Big|\int_{\R^m} g\,d\lambda_t^\gamma-\int_{\R^m} g\,d\lambda\Big|
<2\eps(1+\lambda(D)),
\]
for all $t>T$.
Since $\eps>0$ is arbitrary, this completes the proof that (b) implies (a).
\end{proof}

\section{Stable MLLT and LLD for Gibbs-Markov maps}
\label{app:N1N2}

In this appendix, we show that the key hypotheses (e.g.\ \ref{itm:N1} and \ref{itm:N2} in the fully nonlattice case) are satisfied for a sufficiently large class of observables of the mixing Gibbs-Markov maps defined in Section~\ref{sec:GMflows}.
In Subsection~\ref{app:N2}, we prove local large deviations.
In Subsections~\ref{app:N1} and~{app:N1two} we prove stable MLLTs for one-sided and two-sided Gibbs-Markov maps.

\subsection{Stable local large deviations for Gibbs-Markov maps}
\label{app:N2}

In this subsection, we show how to prove stable LLD for observables $\Psi$ of Gibbs-Markov maps under the appropriate regularity condition~\eqref{eq:Psireg0} below.
This generalises a result of~\cite{MT22} for uniformly piecewise-H\"older observables $\Psi$.

We begin by proving an abstract LLD result extending~\cite[Theorem~3.2]{MT22}.
Let $F:Y\to Y$ be a measure-preserving map on a probability space $(Y,\mu)$.
Fix $\alpha\in(1, 2]$.
Throughout, $\Psi : Y \to\R^m$ is an integrable observation with $\int_Y\Psi\,d\mu=0$ and $\int_Y |\Psi|^2\, d\mu = \infty$. 
Moreover, we suppose that
$\Psi$ is nondegenerate and regularly varying with index $-\alpha$, $\alpha\in(1,2]$, as in Section~\ref{sec:rv}.

Let $R : L^1 \to  L^1$ be the transfer operator for $F$ defined via the formula
$\int_Y Rv\,w\,d\mu=\int_Y v\, w\circ F\,d\mu$.
Let $\cB\subset  L^\infty$ be a Banach space containing constant functions, with
norm $\|\;\|$ satisfying $|v|_\infty \le \|v\|$ for $v\in\cB$.
Since $R(0) = R$ and $\cB$  contains constant functions, $1$ is an eigenvalue of $R(0)$. We
assume:
\begin{itemize}
\item[(1)] The eigenvalue $1$ is simple, and the remainder of the spectrum of $R(0) : \cB\to\cB$ is
contained in a disk of radius less than $1$.
\end{itemize}

Given $s\in\R^m$, 
define the perturbed operator
$R(s):L^1\to L^1$, 
by
$R(s) v = R(e^{is\cdot \Psi} v)$.
Let $\Pi_\eps(x)=\Pi_{q=1}^m (x_q-\eps,x_q+\eps]$ for $x\in\R^m$, $\eps>0$,
and write $\Pi_\eps=\Pi_\eps(0)$.
Let $\partial_q=\partial/\partial s_q$.
We assume
\begin{itemize}
\item[(2)]
There exist $\eps > 0$, $C > 0$, $r\in(0,1)$, such that for all $s, h \in\Pi_\eps$, $q = 1,\dots, m$,
\begin{itemize}
\item[(i)] $\|R(s)\| \leq C$;
\hspace{1em}(ii) $\|\partial_qR(s)\|\leq   C$;
\hspace{1em}(iii) $\|R(s + h) - R(s)\|\leq  C|h|$; 
\item[(iv)] $\|\partial_qR(s + h) - \partial_qR(s)\| \leq  C|h|^{\alpha-1} \tilde\ell(1/|h|)\,(1+|s/h|^r)$.
\end{itemize}
\end{itemize}

\begin{lemma} \label{lem:absLLD}
Assume (1) and (2). Then
for all $h>0$, there exists $C>0$ such that
\[
\mu\{\Psi_n\in\Pi_h(x)\}\le C\frac{n}{a_n^m}\frac{\tilde\ell(|x|)}{1+|x|^\alpha}\Big(1+\Big(\frac{|x|}{a_n}\Big)^r\Big)
\quad\text{for all $n\ge1$, $x\in\R^m$}
\]
where $r\in(0,1)$ is as in (2)(iv).
\end{lemma}

\begin{proof}
We refer to~\cite[Proof of Theorem~3.2]{MT22}, restricting to
the case $\alpha\in(1,2]$ which is the relevant case for us.
The only hypothesis that is changed is~\cite[Hypothesis (H2)(iii)]{MT22} which is modified in accordance with our condition (2)(iv).

The statement of~\cite[Lemma~3.4]{MT22} is unchanged (but of course, its meaning is altered since~\cite[Hypothesis (H2)(iii)]{MT22} is altered).
In the remainder of~\cite[Section~3.2]{MT22}, the only modification required is for the second estimate in~\cite[Lemma~3.7]{MT22}.
We have the revised estimate $\|\partial_qQ(s+h)-\partial_qQ(s)\|
\ll |h|^{\alpha-1}\tilde\ell(1/|h|)(1+(|s|/|h|)^r)$ leading to
\begin{equation} \label{eq:Q}
\|\partial_q(Q(s+h)^n)-\partial_q(Q(s)^n)\|\ll 
\delta_1^n|h|^{\alpha-1}\tilde\ell(1/|h|)(1+(|s|/|h|)^r).
\end{equation}

The proof now proceeds as in~\cite[p.~13]{MT22} until the estimates of the expressions $F_1,\dots,F_6$.
As an example,
\[
F_2 =\frac{n}{|x|}\int_{\Pi_{2\eps}}|\lambda(s_2)|^{n-1} |\partial_q\lambda(s_1)-\partial_q\lambda(s_2)|\|P(s_2)\|\,ds
\]
where $|x|\ge \pi/\eps$, 
$\{s_1,s_2\}=\{s,s-h\}$ and $h\approx 1/|x|$.
Applying the estimates in~\cite{MT22} with the modification from condition~(2)(iv), 
\begin{align*}
F_2 & \ll \frac{n}{|x|}\int_{\Pi_{2\eps}} e^{-c(n-1)|s_2|^\alpha\tilde\ell(1/|s_2|)}|x|^{-(\alpha-1)}\tilde\ell(|x|)(1+|s_2x|^r)\,ds 
\\ & \le \frac{n}{|x|}\int_{\Pi_{3\eps}} e^{-c(n-1)|s|^\alpha\tilde\ell(1/|s|)}|x|^{-(\alpha-1)}\tilde\ell(|x|)(1+|sx|^r)\,ds
\\ & = n\frac{\tilde\ell(|x|)}{|x|^\alpha}\int_{\Pi_{3\eps}} e^{-c(n-1)|s|^\alpha\tilde\ell(1/|s|)}\,ds
+ n\frac{\tilde\ell(|x|)}{|x|^{\alpha-r}}\int_{\Pi_{3\eps}} s^r e^{-c(n-1)|s|^{\alpha}\tilde\ell(1/|s|)}\,ds.
\end{align*}
By~\cite[Lemma~2.3]{MT22}, 
$\int_{\Pi_{3\eps}} s^\beta e^{-c(n-1)|s|^{\alpha}\tilde\ell(1/|s|)}\,ds
\ll a_n^{-(m+\beta)}$ for $\beta\ge0$, so
\[
F_2\ll \frac{n}{a_n^m}\frac{\tilde\ell(|x|)}{|x|^\alpha}
\Big(1+\Big(\frac{|x|}{a_n}\Big)^r\Big).
\]
This estimate is the same as in~\cite{MT22} except for the extra
factor of $1+\big(\frac{|x|}{a_n}\big)^r$.
Similarly, the estimates in~\cite{MT22} for $F_1$ and $F_5$ are unchanged except for this same factor, while the estimates for $F_3$ and $F_4$ are completely unchanged. Finally, using~\eqref{eq:Q},
$F_6\ll \delta_1^n 
\frac{\tilde\ell(|x|)}{|x|^{\alpha-r}}
\ll \frac{n}{a_n^{m+r}}
\frac{\tilde\ell(|x|)}{|x|^{\alpha-r}}$.~
\end{proof}

\begin{thm} \label{thm:LLD}
Let $(F,Y,\mu)$ be a mixing Gibbs-Markov map with associated partition $\beta$ and $\theta\in(0,1)$.
Fix $r\in(0,\alpha-1)$ and $\theta_0\in(0,1)$. 
Suppose that there is a constant $C_0$ such that
\begin{equation} \label{eq:Psireg0}
|\Psi(y)-\Psi(y')|\le C_0(\infb |\Psi|+1)d_{\theta_0}(y,y')
\quad\text{for all $y,y'\in b$, $b\in\beta$}.
\end{equation}
Then for all $h>0$, there is a constant $C$ such that
\[
\mu\{\Psi_n\in\Pi_h(x)\}\le C\frac{n}{a_n^m}\frac{\tilde\ell(|x|)}{1+|x|^\alpha}\Big(1+\Big(\frac{|x|}{a_n}\Big)^r\Big)
\quad\text{for all $n\ge1$, $x\in\R^m$}.
\]
\end{thm}

\begin{proof}
Without loss, we can suppose that $h\in(0,1)$.
Increase $\theta$ if necessary so that $\theta\in(\theta_0^r,1)$.
Given $v:Y\to\R^d$,  we let
\[
D_{b,\theta}\,v=\sup_{y,y'\in b,\,y\neq y'}|v(y)-v(y')|/d_\theta(y,y'),
\qquad |v|_\theta=\sup_{b\in\beta}D_{b,\theta}v.
\]
We take $\cB=\cF_\theta$ to be the Banach space of functions $v:Y\to\R$ such that $|v|_\theta<\infty$ with norm $\|v\|_\theta=|v|_\infty+|v|_\theta$.
It remains to verify conditions (1) and~(2) in Lemma~\ref{lem:absLLD}.

Since $R(0)$ does not depend on $\Psi$, condition~(1) is unchanged from~\cite{MT22}. We focus here on conditions (2)(ii) and (2)(iv); the other parts of condition~(2) are simpler.

As in the proof of~\cite[Proposition~3.10]{MT22}, we note that
\[
\|R(v\chi)\|_\theta \ll \|v\|_\theta\sum_{b\in\beta} \mu(b)(\supb|\chi|+D_{b,\theta}\chi)
\]
for $\chi:Y\to\C$.
Hence
$\|\partial_q R(s)\|\ll \sum_{b\in\beta} \mu(b)(\supb|\chi|+D_{b,\theta}\chi)$ with 
$\chi=\Psi_q e^{is\cdot \Psi}$.
Now $\supb|\chi|=\supb|\Psi|_q\ll \infb|\Psi|+1$
and $D_{b,\theta}\psi\le D_{b,\theta_0}\psi\ll\infb|\Psi|+1$. 
For $y,y'\in b$,
we have
\begin{align*}
|e^{is\cdot \Psi(y)}-e^{is\cdot \Psi(y')}| & \le \min\{|s||\Psi(y)-\Psi(y')|,2\}
 \le 2|s|^r|\Psi(y)-\Psi(y')|^r
\\ & \ll |s|^r (\infb|\Psi|^r+1) d_{\theta_0}(y,y')^r
\le |s|^r(\infb|\Psi|^r+1) d_\theta(y,y'),
\end{align*}
so $D_{b,\theta}e^{is\cdot \Psi}\ll |s|^r(\infb|\Psi|^r+1)$.
It follows that
\[
\|\partial_q R(s)\|\ll \sum_{b\in\beta} \mu(b)(\infb|\Psi|+\infb|\Psi|^{1+r}+1)\le |\Psi|_1+|\Psi|_{1+r}^{1+r}+1<\infty
\]
by the choice $r<\alpha-1$.
This completes the proof of (2)(ii).

Turning to (2)(iv), 
$\|\partial_qR(s+h)-\partial_qR(s)\|\ll \sum_{b\in\beta} \mu(b)(\supb|\chi|+D_{b,\theta}\chi)$ with
$\chi=\Psi_q e^{is\cdot \Psi}(e^{ih\cdot \Psi}-1)$. 
As in~\cite{MT22}, there are four terms to estimate, the most difficult of which is
\[
S=\sum_{b\in\beta} 
 \supb |\Psi_q|\, \supb |e^{ih\cdot \Psi}-1|\,D_{b,\theta}e^{is\cdot \Psi}.
\]
We have already seen that
$\supb |\Psi_q| \ll \infb|\Psi|+1$ and $D_{b,\theta}e^{is\cdot \Psi}\ll
|s|^r (\infb|\Psi|^r+1)$.
Also, $\supb |e^{ih\cdot \Psi}-1|\le 
\min\{h\supb|\Psi|,2\}\ll
\min\{h(\infb|\Psi|+1),1\}$.
Hence
\begin{align*}
S& \ll 
|s|^r \sum_{b\in\beta} \mu(b)(\infb|\Psi|^{1+r}+1) \min\{h(\infb|\Psi|+1),1\}
\\ & \le |s|^r \int_Y (|\Psi|^{1+r}+1) \min\{h(|\Psi|+1),1\}\,d\mu.
\end{align*}

Let $G(x)=\mu(|\Psi|\le x)$ denote the distribution function of $|\Psi|$, so
$1-G(x)=x^{-\alpha}\ell(x)$ for $x>0$.
Then for $K>1$,
\[
S\ll  |s|^r \int_0^\infty (x^{1+r}+1) \min\{h(x+1),1\}\,dG(x)
\ll  |s|^r(hI_1+I_2)
\]
where
$I_1=\int_0^K (x^{2+r}+1) \,dG(x)$ and
$I_2=\int_K^\infty x^{1+r}\,dG(x)$.
Since $r\in(0,\alpha-1)$, we can proceed as in the proof of~\cite[Lemma~2.2(ii)]{MT22} to obtain
$I_1\ll K^{2-\alpha+r}\ell(K)$ and 
$I_2\ll K^{1-\alpha+r}\ell(K)$.
Taking $K\approx 1/h$, we obtain the estimate
$S \ll |s|^r\ell(1/h)h^{\alpha-1-r}$.
The other three terms in (2)(iv) are more easily estimated using the ideas seen already and are bounded by
$\tilde\ell(1/h)h^{\alpha-1}$.
Hence we obtain the desired estimate in (2)(iv).
\end{proof}

\subsection{Stable MLLT for Gibbs-Markov maps}
\label{app:N1}

In this subsection, we prove stable MLLTs for mixing Gibbs-Markov maps.
For simplicity of notation, we write $A$, $B$ instead of $A_0$, $B_0$.
We consider observables with values in $\R^{m_1}\times\Z^{m_2}$
covering simultaneously 
the nonlattice case ($m_1=d+1$, $m_2=0$) in \ref{itm:N1},
the lattice case ($m_1=0$, $m_2=d+1$) in \ref{itm:L1},
and the hybrid case ($m_1=d$, $m_2=1$) in \ref{itm:NL1}.

We note that analogous proofs have been provided for $m=1$ and in some special cases for $m\ge2$ in~\cite{AD01,AD01b} (though for a smaller class of observables).
We follow the argument in~\cite{AD01} which is based on~\cite{Breiman}.
For i.i.d.\ random variables, the hybrid case was considered in~\cite{Doney91}.

Our underlying assumptions are as follows.
As in Subsection~\ref{app:N2},
we suppose that $(F,Y,\mu)$ is a mixing Gibbs-Markov map with associated partition $\beta$ and $\theta\in(0,1)$.

Let
$\Psi=(u,v):Y\to \R^{m_1}\times \Z^{m_2}$ be integrable but not square integrable
with $\int_Y u\,d\mu=0$.
We suppose that $\Psi$ is aperiodic, nondegenerate, regularly varying with index $-\alpha$, $\alpha\in(1,2]$, and assume that $\Psi$ satisfies~\eqref{eq:Psireg0}.

Let $m=m_1+m_2$ where $m_1,m_2\ge0$ and $m\ge1$.
Write $z\in\R^m$ as $z=(z_1,z_2)$ where $z_j\in\R^{m_j}$.
Similarly, write $s=(s_1,s_2)$.

Let $g_{X_\alpha}$ denote the density of the limiting stable law $X_\alpha$ 
(so
$g_{X_\alpha}(x)=(2\pi)^{-m}\int_{\R^m}e^{-is\cdot x}\,\E \, e^{is\cdot X_\alpha}\,ds$).
Define $\cF_\theta$ as in Subsection~\ref{app:N2}.
Let 
$\hv^{(n)}= v_n-\big[n\int_Y v\,d\mu\big]$.

\begin{thm} \label{thm:LLT}
Let $\zeta_1\in \cF_\theta$, $\zeta_2\in L^1$.  
Let $K\subset\R^{m_1}$ be bounded with $|\partial K|=0$, and let $L>0$.
Then
\begin{align*}
\lim_{n\to\infty}
\sup_{\genfrac{}{}{0pt}{}{z\in\R^m}{|z|\le L}}
\Big| a_n^m\, \int_Y\zeta_1\;\zeta_2\circ F^n\; & 1_{\{u_n\in z_1a_n+K,\,\hv^{(n)}=[z_2a_n]\}} \,d\mu
-g_{X_\alpha}(z)\int_Y\zeta_1\,d\mu\int_Y\zeta_2\,d\mu\Big|=0.
\end{align*}
\end{thm}

In the remainder of this subsection we prove Theorem~\ref{thm:LLT}.

Let $\tPsi=\big(u,v-\int_Y v\,d\mu\big)$ and
define $R(s)v=R(e^{is\cdot \tPsi}v)$.
Again, $R(0)1=1$ and $1$ is a simple eigenvalue for $R(0)$ by ergodicity.
As before, 
there exist $\eps>0$ and a family of eigenvalues $\lambda(s)$ for
$R(s)$ with corresponding spectral projections $P(s):\cF_\theta\to\cF_\theta$
such that $\lambda(0)=1$ and $P(0)v=\int_Y v\,d\mu$.

\begin{prop} \label{prop:R}
Let $\alpha'\in(0,\alpha)$.
There exist $\eps,\,c,\,C>0$, $\delta_0\in(0,1)$, such that
\begin{itemize}
\item[(a)] $R(s)=\lambda(s)P(s)+Q(s)$, \hspace{1em}
(b) $\|P(s)-P(0)\|\le C|s|$,
\item[(c)] $\|Q(s)\|^n\le C\delta_0^n$, \hspace{4.5em}
(d) $|\lambda(s)|\le \exp\{-c|s|^\alpha \tilde\ell(1/|s|)\}$,
\end{itemize}
for $|s|\le \eps$, $n\ge1$, and
\begin{itemize}
\item[(e)] 
$|\lambda(a_n^{-1}s)|^n\le Ce^{-c|s|^{\alpha'}}$ for $|s|\le \eps a_n$
and $n$ sufficiently large,
\item[(f)] $\lim_{n\to\infty}\lambda(a_n^{-1}s)^n= \E(e^{is\cdot X_\alpha})$
for all $s\in\R^m$.
\end{itemize}
\end{prop}

\begin{proof}
Properties (a)--(d) are unchanged from Subsection~\ref{app:N2}.
Since $a_n^\alpha\sim n\tilde\ell(a_n)$, it follows from~(d) that
$|\lambda(a_n^{-1}s)|^n\le \exp\{-nc|a_n^{-1}b|^\alpha \tilde\ell(a_n/|s|)\}
\le \exp\{-\frac12 c |s|^\alpha \tilde\ell(a_n/|s|)\tilde\ell(a_n)^{-1}\}$
for $n$ large. By Potter's bounds, there exists $c'$ such that
$|\lambda(a_n^{-1}s)|^n\le \exp\{-c'|s|^{\alpha'}\}$ for $n$ large,
proving~(e).

To prove part (f), 
first define the characteristic function
$\phi(s)=\int_Y e^{is\cdot\Psi}\,d\mu$, $s\in\R^m$.
Since $\Psi$ is in the domain of the stable law $X_\alpha$, we have that
$\lim_{n\to\infty}\phi(a_n^{-1}s)^n= \E(e^{is\cdot X_\alpha})$.

Following the proof of~\cite[Lemma~3.5]{MT22}, for $|s|<\eps$,
\[
\lambda(s)=\phi(s)+H(s)
\quad\text{where} \quad H(s)=O(|s|^2).
\]
Hence
\(
\lambda(s)= \phi(s)(1+\tH(s))
\)
where $\tH(s)=H(s)\phi(s)^{-1}$.
Shrinking $\eps$ so that $|\phi(s)|{>\frac12}$ for $|s|<\eps$,
we have $\tH(s)=O(|s|)^2$. 

Now fix $s\in\R^m$ and consider $n$ large so that $a_n^{-1}|s|<\eps$. Then 
\[
n\tH(a_n^{-1}s)\ll na_n^{-2}\sim n^{1-2/\alpha}\tilde\ell(a_n)^{-2/\alpha}
=o(1),
\]
 so $\lim_{n\to\infty}n\tH(a_n^{-1}s)=0$.
Hence,
\[
\lambda(a_n^{-1}s)^n=
\phi(a_n^{-1}s)^n(1+\tH(a_n^{-1}s))^n
\sim \phi(a_n^{-1}s)^n \exp\{n\tH(a_n^{-1}s)\}
\to \E(e^{is\cdot X_\alpha})
\]
as $n\to\infty$, completing the proof of~(f).
\end{proof}

\begin{prop} \label{prop:ap}
For any $M>\eps>0$,  there exist $C>0$, $\delta_0\in(0,1)$ such that
\[
\|R(s)^n\|\le C\delta_0^n
\quad\text{for $\eps\le |s|\le M$.}
\]
\end{prop}

\begin{proof}
It is standard that the 
spectral radius of $R(s)$ is at most $1$ and
the essential spectral radius of $R(s)$ is less than $1$ for all $s$.
We claim that there are no eigenvalues on the unit circle for $s\neq0$.
Then the spectral radius is less than $1$ for all $s\neq0$ and the result follows from
continuity of $s\mapsto R(s)$.

It remains to prove the claim. Suppose that there exist
$\lambda\in \bbS^1$ and $v\in\cF_\theta$ such that
$R(s)v=\lambda v$.
Define $U_s:L^2\to L^2$, $U_sv=e^{-is\cdot \tPsi}v\circ F$.
Then $U_s$ is the $L^2$ adjoint of $R(s)$ and a calculation shows that
$\|U_sv-\bar\lambda v\|_2=0$.
Hence $v\circ F=\bar\lambda e^{is\cdot \tPsi}v$.
In particular, $|v|\circ F=|v|$, so $|v|\equiv1$ by ergodicity.
Since $s\neq0$, this violates the aperiodicity assumption. We conclude that there are no eigenvalues on the unit circle, verifying the claim.
\end{proof}

The next result is formulated so as to facilitate the proof of the stable MLLT in the two sided case (see Subsection~\ref{app:N1two} below).

\begin{lemma} \label{lem:LLT}
Let $g\in C_c(\R^m)$ and define
\begin{align*}
G_{n,z} & =
 \int_{\R^m} g(s)\int_Y \zeta_1(s)\;\zeta_2(s)\circ F^n e^{is\cdot \tPsi_n}\,d\mu\,ds, \\ 
 I_{n,z} & =
 \int_{\R^m}1_{\{|s|< \eps }\} g(s)
\Big(\int_Y \zeta_1(s)\,d\mu\Big)
\Big(\int_Y \zeta_2(s)\,d\mu\Big)
\lambda(s)^n  \,ds.
\end{align*}

There is a constant $C>0$ such that
\[
|G_{n,z}- I_{n,z}|
\le C\big({\textstyle\sup_s}\|\zeta_1(s)\|\,|\zeta_2(s)|_1\big)
a_n^{-(m+1)}, 
\]
for all $\zeta_1:\R^m\to \cF_\theta$, $\zeta_2:\R^m\to  L^1$, $n\ge1$. 
\end{lemma}

\begin{proof}
Note that
\[
G_{n,z}
 =\int_{\R^m}g(s)\int_Y \zeta_2(s)\, R(s)^n\zeta_1(s) \,d\mu\,ds.
\]

Let $\eps>0$ be as in Proposition~\ref{prop:R} and
choose $M>\eps$ so that $\supp g\subset B_M(0)$.
By Proposition~\ref{prop:R}(a), for $|s|\le\eps$,
\[
R(s)^n=\lambda(s)^nP(0)+\lambda(s)^n(P(s)-P(0))+Q(s)^n.
\]
Hence
\(
G_{n,z}= I_1+I_2+I_3+I_4
\)
where
\begin{align*}
I_1 & = \int_{\R^m}
1_{\{|s|<\eps\}} g(s)\int_Y\zeta_2(s) \lambda(s)^n P(0) \zeta_1(s)\,d\mu\,ds
\\ & = \int_{\R^m}
1_{\{|s|<\eps\}} g(s)\mu(\zeta_1(s))\mu(\zeta_2(s))\lambda(s)^n \,ds, 
\\
I_2 & = \int_{\R^m}
1_{\{|s|<\eps\}}g(s)
\int_Y \zeta_2(s)\, \lambda(s)^n \big(P(s)-P(0)\big) \zeta_1(s)
\,d\mu\,ds, \\
I_3 & = \int_{\R^m}
1_{\{|s|<\eps\}}g(s)
\int_Y \zeta_2(s)\, Q(s)^n \zeta_1(s)
\,d\mu\,ds, \\
I_4 & = \int_{\R^m}
1_{\{\eps\le |s|\le M\}}g(s)
\int_Y \zeta_2(s)\, R(s)^n \zeta_1(s) \,d\mu\,ds.
\end{align*}

By Proposition~\ref{prop:ap},
\begin{align*}
|I_4|  
& \le |g|_\infty\, \int_{\R^m}
1_{\{\eps\le |s| \le M\}}
|\zeta_2(s)|_1\, |R(s)^n\zeta_1(s)|_\infty \,ds \\ & 
\le |g|_\infty\, \int_{\R^m}
1_{\{\eps\le |s| \le M\}}
|\zeta_2(s)|_1\, \|R(s)^n\| \|\zeta_1(s)\|\,ds
\ll \kappa\delta_0^n,
\end{align*}
where $\kappa= {\textstyle \sup_s}\|\zeta_1(s)\| \,|\zeta_2(s)|_1$.
Similarly, by Proposition~\ref{prop:R}(c), $|I_3|\ll   
 \kappa \delta_0^n$.
Also,
\begin{align*}
|I_2| & \ll  
 \kappa 
\int_{\R^m} 1_{\{|s|< \eps\}}
   |\lambda(s)|^n\|P(s)-P(0)\|\,ds
\ll  \kappa\int_{\R^m}\exp\{-nc|s|^\alpha \tilde\ell(1/|s|)\} |s|\,ds
\end{align*}
by Proposition~\ref{prop:R}(b,d). A change of variables (\cite[Lemma~2.3]{MT22})
now shows that $|I_2|\ll \kappa a_n^{-(m+1)}$.
\end{proof}

\begin{pfof}{Theorem~\ref{thm:LLT}}
Define the family of measures on $\R^{m_1}$,
\[
\nu_{n,z}(K)=a_n^m\,\int_Y \zeta_1\;\zeta_2\circ F^n\;1_{\{u_n\in z_1a_n+K\}}
1_{\{\hv^{(n)}=[z_2a_n]\}}\,d\mu.
\]
Also define $\nu_z(K)=g_{X_\alpha}(z)
\int_Y\zeta_1\,d\mu\, \int_Y\zeta_2\,d\mu\, |K|$.
We claim that $\lim_{n\to\infty,\,w\to z}\int_{\R^{m_1}}h\,d\nu_{n,w}=\int_{\R^{m_1}} h\,d\nu_z$
for all $h\in L^1(\R^{m_1})$ with compactly supported Fourier transform $\hat h$.
It then follows from~\cite[Theorem~10.7]{Breiman} that 
$\lim_{n\to\infty,\,w\to z}\nu_{n,w}(K)\to\nu_z(K)$ for all $K\subset\R^{m_1}$ bounded with $|\partial K|=0$. 
Now fix such a $K$. It follows easily that 
$\lim_{n\to\infty}\nu_{n,w}(K)\to\nu_w(K)$ uniformly for $w$ in a neighbourhood of $z$ 
for each fixed $z\in\R^m$. By compactness, we obtain uniform convergence for $|z|\le L$.

It remains to verify the claim.
Note that
$h(x)=\int_{\R^{m_1}}e^{is_1\cdot x}\hat h(s_1)\,ds_1$ and
\begin{align*}
1_{\{\hv^{(n)}=[w_2a_n]\}}
& =(2\pi)^{-m_2}\int_{[-\pi,\pi]^{m_2}}e^{-is_2\cdot [w_2a_n]}e^{is_2\cdot \hv^{(n)}}\,ds
\\ & =
(2\pi)^{-m_2}\int_{[-\pi,\pi]^{m_2}}e^{-is_2\cdot([w_2a_n]+c_n)}e^{is_2\cdot \tv_n}\,ds,
\end{align*}
where $\tv=v-\int_Y v\,d\mu$ and
$c_n=\big[n\int_Y v\,d\mu\big]-n\int_Y v\,d\mu$.
By Fubini,
\begin{align*}
& \int_{\R^{m_1}}h\,d\nu_{n,w}
 =a_n^m\int_Y \zeta_1\;\zeta_2\circ F^n\, h(u_n-w_1a_n)1_{\{\hv^{(n)}=[w_2a_n]\}}\,d\mu
\\ & \quad =(2\pi)^{-m_2}a_n^m\int_{\R^{m_1}\times[-\pi,\pi]^{m_2}}e^{-i s_1\cdot w_1a_n}e^{-is_2\cdot([w_2a_n]+c_n)}\hat h(s_1)\int_Y \zeta_1\;\zeta_2\circ F^n\, e^{is\cdot \tPsi_n}\,d\mu\,ds.
\end{align*}
Let $\zeta=(2\pi)^{-m_2}\mu(\zeta_1)\mu(\zeta_2)$.
Applying Lemma~\ref{lem:LLT}
with 
\[
g(s)=e^{-i s_1\cdot w_1a_n}e^{-is_2\cdot([w_2a_n]+c_n)}\hat h(s_1)1_{\{s_2\in[-\pi,\pi]^{m_2}\}},
\]
we obtain
\begin{align*}
\int_{\R^{m_1}}h\,d\nu_{n,w}
& =a_n^m \zeta
 \int_{|s|< \eps }e^{-is_1\cdot w_1a_n}e^{-is_2\cdot([w_2a_n]+c_n)} \hat h(s_1)\lambda(s)^n\, ds +O(a_n^{-1})
\\ & = 
\zeta
\int_{|s|< a_n \eps }e^{-is_1\cdot w_1}e^{-ia_n^{-1}s_2\cdot([w_2a_n]+c_n)} \hat h(a_n^{-1}s_1) \lambda(a_n^{-1}s)^nds + O(a_n^{-1}).
\end{align*}
By Proposition~\ref{prop:R}(e,f),
it follows from the dominated convergence theorem that
\[
\lim_{n\to\infty,\,w\to z}  
\int_{\R^{m_1}}h\,d\nu_{n,w} = 
\zeta
\int_{\R^m}\hat h(0)\,e^{-is\cdot z}\, \E(e^{is\cdot X_\alpha})\,ds.
\]
Tracking through the definitions,
\[
\int_{\R^m}\hat h(0)\,e^{-is\cdot z}\, \E(e^{is\cdot X_\alpha})\,ds=
\hat h(0)(2\pi)^m g_{X_\alpha}(z)
=(2\pi)^{m_2}\int_{\R^{m_1}}h(x)\,dx\,g_{X_\alpha}(z).
\]
Hence, we have shown that
\[
\lim_{n\to\infty,\,w\to z} \int_{\R^{m_1}}h\,d\nu_{n,w}=
g_{X_\alpha}(z)
\mu(\zeta_1)\mu(\zeta_2)
\int_{\R^{m_1}} h\,dx =\int_{\R^{m_1}}h\,d\nu_z,
\]
as required.
\end{pfof}

\begin{prop} \label{prop:A}
Let $A\subset Y$ be measurable with $\mu(\partial A)=0$.
For any $\eps>0$, there exist $A_1,A_2\subset Y$ with $1_{A_1},\,1_{A_2}\in\cF_\theta$ such that $A_1\subset A\subset A_2$ and $\mu(A_2)-\mu(A_1)<\eps$.
\end{prop}

\begin{proof}
Let $\beta_n$ be the set of closures of $n$-cylinders in $Y$
and let $\beta=\bigcup_{n\ge1}\beta_n$.
Using the greedy algorithm, we write $\Int A$ as a disjoint countable union $\Int A=\bigcup_{j\ge 1}b_j$ where $b_j\in \beta$.
(At the $j$'th step, choose $b_j\in\beta$ such that $\mu(b_j)\ge\frac12\sup\{b\in \beta:b\subset \Int A\setminus (b_1\cup\dots\cup b_{j-1})$.)
Setting $A_1=\bigcup_{j=1}^J b_j$ we obtain $A_1\subset A$ such that $1_{A_1}\in\cF_\theta$ and $\mu(A)-\mu(A_1)<\eps/2$.
Similarly, we construct $D\subset \Int A^c$ such that $1_D\in\cF_\theta$ and $\mu(A^c)-\mu(D)<\eps/2$.
Let $A_2=D^c$. Then $1_{A_2}=1_Y- 1_D\in\cF_\theta$
and $\mu(A_2)-\mu(A)<\eps/2$.~
\end{proof}

\begin{cor} \label{cor:LLT}
Let $A,B\subset Y$ be measurable sets with $\mu(\partial A)=0$.
Let $K\subset\R^{m_1}$ be bounded with $|\partial K|=0$, and let $L>0$. Then
\begin{align*}
\lim_{n\to\infty} \sup_{\genfrac{}{}{0pt}{}{z\in\R^m}{|z|\le L}}
\Big|a_n^m\mu\big(y\in A\cap  F^{-n}B:\, u_n(y)\in z_1a_n + & K,\,
\hv^{(n)}(y)=[z_2a_n]\big)
\\ & - g_{X_\alpha}(z)\mu(A)\mu(B)|K|\Big|=0.
\end{align*}
\end{cor}

\begin{proof}
The result is immediate from Theorem~\ref{thm:LLT} if we suppose in addition that
$1_A\in\cF_\theta$, so it remains to relax this condition.
For this step, we can regard $B$ and $K$ as fixed.
Let $A\subset Y$ be measurable with $\mu(\partial A)=0$ and let $\eps>0$.
Choose $A_1,A_2$ as in Proposition~\ref{prop:A}.
Let 
\[
q_{n,z}(E)=a_n^m\,\mu(y\in E\cap F^{-n}B:u_n(y)\in z_1a_n+K,\,\hv^{(n)}(y)=z_2a_n)
\]
and set $c=g_{X_\alpha}(z)\mu(B)|K|$.
Then
\begin{align*}
c\mu(A_1) & =\lim_{n\to\infty,\,w\to z} q_{n,w}(A_1)
\le \liminf_{n\to\infty,\,w\to z} q_{n,w}(A)
\\ &\le \limsup_{n\to\infty,\,w\to z} q_{n,w}(A)
 \le \lim_{n\to\infty,\,w\to z} q_{n,w}(A_2)=c\mu(A_2).
\end{align*}
Hence
$c(\mu(A)-\eps)\le \liminf_{n\to\infty,\,w\to z} q_{n,w}(A)
\le \limsup_{n\to\infty,\,w\to z} q_{n,w}(A)
\le c(\mu(A)+\eps)$.
Since $\eps>0$ is arbitrary,
$\lim_{n\to\infty,\,w\to z}q_{n,w}(A)=c\mu(A)$ and the result follows.
\end{proof}

An alternative reformulation of Corollary~\ref{cor:LLT} is as follows:

\begin{cor}  \label{cor:LLTvague}
Let $\zeta_1,\,\zeta_2:Y\to\R$
be bounded almost every continuous functions,
$\xi\in C_c(\R^m)$, $L>0$. Then
\begin{align*}
\lim_{n\to\infty} \sup_{\genfrac{}{}{0pt}{}{z\in\R^m}{|z|\le L}}
a_n^m\int_Y \zeta_1\;\zeta_2\circ F^n &  \;\xi\circ(u_n-z_1a_n,\hv^{(n)}]-[z_2a_n])\,d\mu
\\ & \qquad =g_{X_\alpha}(z)\int_Y \zeta_1\,d\mu
\int_Y \zeta_2\,d\mu
\int_{\R^m} \xi\,d\Leb.
\end{align*}
\end{cor}

\begin{proof}
The reformulation in terms of $\xi$ for the $\R^m$-variables follows directly from Lemma~\ref{lem:Port}. The reformulation in terms of $\zeta_1,\zeta_2$ for the $Y$-variables works the same way. The argument that (b) implies (a) in the proof of Lemma~\ref{lem:Port} does not use local compactness and is in fact easier here since there is no need to introduce the bounded set $D$.
\end{proof}

\subsection{Stable MLLT for two-sided Gibbs-Markov maps}
\label{app:N1two}

In this subsection, we consider the stable MLLT for mixing two-sided Gibbs-Markov maps $(F,Y,\mu)$.  

Assume that $\Psi=(u,v):Y\to\R^{m_1}\times\Z^{m_2}$ is integrable but not square integrable with $\int_Y u\,d\mu=0$.
Let $\hv^{(n)}= v_n -\big[n\int_Y v\,d\mu\big]$.

\begin{thm} \label{thm:LLTtwo}
Let $(F,Y,\mu)$ be a mixing two-sided Gibbs-Markov map with one-sided quotient
$\bpi:Y\to\bY$, and
let $\cA$ denote the collection of measurable subsets $A\subset Y$ with $\mu(\partial A)=0$.

Assume that $\Psi=(u,v):Y\to\R^{m_1}\times\Z^{m_2}$ is aperiodic, nondegenerate, regularly varying with index $-\alpha$,  $\alpha\in(1,2]$.
Suppose further that $v$ is locally constant and that $u$ admits a decomposition of the form
\begin{equation} \label{eq:chi0}
u=\bu\circ\bar\pi+\chi\circ F-\chi,
\end{equation}
where $\bu:\bY\to\R^{m_1}$ satisfies~\eqref{eq:Psireg0} and
$\chi|_A:A\to\R^{m_1}$ is uniformly continuous
for each $A\in\cA$.

Then
\begin{align*}
\lim_{n\to\infty} \sup_{\genfrac{}{}{0pt}{}{z\in\R^m}{|z|\le L}}
\Big|a_n^m\mu\big(y\in A\cap  F^{-n}B:\, u_n(y) \in z_1a_n +K, &
\, \hv^{(n)}(y)=[z_2a_n]   \big)
\\ & - g_{X_\alpha}(z)\mu(A)\mu(B)|K|\Big|=0,
\end{align*}
for all $A,B\in\cA$,
$K\subset\R^{m_1}$ bounded with $|\partial K|=0$, and $L>0$.
\end{thm}

\begin{proof}
Fix $\gamma\in(0,1)$.
Recall that a test function $\zeta:Y\to\R$ is dynamically H\"older if $\zeta$ is bounded and there exist constants $C>0$, $\gamma\in(0,1)$ such that $|\zeta(F^ny)-\zeta(F^ny')|\le C(\gamma^{s(y,y')}+\gamma^{s(y,y')-n})$ for all
$y,y'\in Y$ and $n\ge0$.

First, we prove the MLLT with $1_A$ and $1_B$ replaced by dynamically H\"older test functions
$\zeta_1,\zeta_2:Y\to\R$.
Define
\[
\zeta_{r,k}:Y\to\R, \qquad \zeta_{r,k}(y)=\inf\{\zeta_r(F^ky'):s(y,y')\ge 2k\},
\]
for $r=1,2$ and $k\ge1$.
Then $\zeta_{r,k}=\bzeta_{r,k}\circ \bpi$ where
$\bzeta_{r,k}\in L^\infty(\bY)$ with
$|\bzeta_{r,k}|_\infty\le |\zeta_r|_\infty$. It is standard that
$|\zeta_r\circ F^k-\zeta_{r,k}|_\infty\ll \gamma^k$. Moreover, we can choose $\gamma$ so that
$\sup_k\|R^k\bzeta_{r,k}\|<\infty$.

Let $\tu_n=\bu_n\circ\bpi$. Since $v$ is locally constant, there exists $\bv^{(n)}$ such that
$\hv^{(n)}=\bv^{(n)}\circ\bpi$. Then 
\begin{align*}
\int_Y & \zeta_1\;\zeta_2\circ F^n \;1_{\{\tu_n\in z_1a_n+K,\,\hv^{(n)}=[z_2a_n]\}}\,d\mu
\\ & =
\int_Y \zeta_1\circ F^k\;\zeta_2\circ F^{n+k} \,1_{\{\tu_n\in z_1a_n+K,\,\hv^{(n)}=[z_2a_n]\}}\circ F^k\,d\mu
 = J_1+J_2+J_3,
\end{align*}
where
\begin{align*}
J_1 & = 
\int_Y (\zeta_1\circ F^k-\zeta_{1,k})\;\zeta_2\circ F^{n+k} 
\,1_{\{\tu_n\in z_1a_n+K,\,\hv^{(n)}=[z_2a_n]\}}\circ F^k\,d\mu
\ll \gamma^k, \\
J_2 & = 
\int_Y \zeta_{1,k}\;(\zeta_2\circ F^k-\zeta_{2,k})\circ F^n 
\,1_{\{\tu_n\in z_1a_n+K,\,\hv^{(n)}=[z_2a_n]\}}\circ F^k\,d\mu
\ll \gamma^k, \\
J_3 & = 
\int_Y \zeta_{1,k}\;\zeta_{2,k}\circ F^n 
\,1_{\{\tu_n\in z_1a_n+K,\,\hv^{(n)}=[z_2a_n]\}}\circ F^k\,d\mu
 \\ & =
\int_\bY \bzeta_{1,k}\;\bzeta_{2,k}\circ \bF^n 
\,1_{\{\bu_n\in z_1a_n+K,\,\bv^{(n)}=[z_2a_n]\}}\circ \bF^k\,d\bmu
 \\ & =
\int_\bY \big(R^k\bzeta_{1,k}\big)\;\bzeta_{2,k}\circ \bF^{n-k} 
\,1_{\{\bu_n\in z_1a_n+K,\,\bv^{(n)}=[z_2a_n]\}}\,d\bmu.
\end{align*}
We take $k=q\log n$ where $q$ is sufficiently large that $\lim_{n\to\infty}a_n^m\gamma^{q\log n}=0$.
Then
\[
a_n^m\int_Y \zeta_1\;\zeta_2\circ F^n \;1_{\{\tu_n\in z_1a_n+K,\,\hv^{(n)}=[z_2a_n]\}}\,d\mu
= \nu_{n,z}(K)+o(1)
\]
as $n\to\infty$ uniformly in $z$, where
\[
\nu_{n,z}(K)=a_n^m\int_\bY \big(R^k\bzeta_{1,k}\big)\;\bzeta_{2,k}\circ \bF^{n-k} 
\,1_{\{\bu_n\in z_1a_n+K,\,\bv^{(n)}=[z_2a_n]\}}\,d\bmu.
\]

Let $D=\R^{m_1}\times[-\pi,\pi]^{m_2}\times\bY$ and $q=n-k$.
Proceeding as in the proof of Theorem~\ref{thm:LLT},
\begin{align*}
(2\pi)^{m_2} & \int_{\R^{m_1}}h \,d\nu_{n,w}  \\ & =
a_n^m\int_De^{-is_1\cdot w_1a_n}
e^{-is_2\cdot([w_2a_n]+c_n)}\hat h(s_1)
\big(R^k\bzeta_{1,k}\big)\;\bzeta_{2,k}\circ \bF^q e^{is\cdot \oPsi_n}
\,d\bmu\,ds
 \\ &  =
a_n^m\int_{D}e^{-is_1\cdot w_1a_n}e^{-is_2\cdot([w_2a_n]+c_n)}\hat h(s_1)
\big(R^k\bzeta_{1,k}\big)\;\big(\bzeta_{2,k}e^{is\cdot\oPsi_k}\big)\circ \bF^q e^{is\cdot \oPsi_q}\,
d\bmu\,ds,
\end{align*}
where $\bPsi=\big(\bu,v-\int_Y v\,d\mu\big)$
and $c_n=\big[n\int_Y v\,d\mu\big]-n\int_Y v\,d\mu$.
Hence, by Lemma~\ref{lem:LLT},
\begin{align*}
\int_{\R^m}h\,d\nu_{n,w} & =
 I_1 + O(\|R^k\bzeta_{1,k} \|\,{\textstyle\sup_s}|\bzeta_{2,k}e^{is\cdot\oPsi_k}  |_1 a_n^m a_q^{-(m+1)})=I_1+O(a_n^{-1})
\end{align*}
where
\begin{align*}
I_1  
 & = a_n^m \bmu\big(R^k\bzeta_{1,k}\big)\,
\int_{|s|< \eps }
e^{-is_1\cdot w_1a_n}e^{-is_2\cdot([w_2a_n]+c_n)}\hat h(s) 
\bmu\big(\bzeta_{2,k}e^{is\cdot\oPsi_k}\big)\,
\lambda(s)^q ds
 \\ & = \bmu\big(R^k\bzeta_{1,k}\big)\,
\int_{|s|< a_n \eps }
e^{-is_1\cdot w_1}e^{-ia_n^{-1}s_2\cdot([w_2a_n]+c_n)}\hat h(a_n^{-1}s) 
\bmu\big(\bzeta_{2,k}e^{ia_n^{-1}s\cdot\oPsi_k}\big)\,
\lambda(a_n^{-1}s)^q ds.
\end{align*}
But
\begin{align*}
\bmu\big(R^k\bzeta_{1,k}\big) & =
\mu\big( \zeta_{1,k}\big)=\mu\big( \zeta_1\big)+O(\gamma^k), \\
\bmu\big( \bzeta_{2,k}e^{is\cdot\oPsi_k}\big) & =
\mu\big( \zeta_{2,k}e^{is\cdot\tPsi_k}\big)=
\mu\big( \zeta_2\circ F^k e^{is\cdot\tPsi_k}\big)+O(\gamma^k).
\end{align*}
Hence
\begin{align*}
\lim_{n\to\infty,\,w\to z}
\int_{\R^m}h\,d\nu_{n,w}  & =
\mu(\zeta_1)\,\mu(\zeta_2)\,
\int_{\R^m}\hat h(0)e^{-s\cdot x}\E(e^{is\cdot X_\alpha})\,ds
\\ & =
\mu(\zeta_1)\,\mu(\zeta_2)\,
\int_{\R^m}h(x)\,dx\,g_{X_\alpha}(z).
\end{align*}
As in the proof of Theorem~\ref{thm:LLT}, it follows from~\cite[Theorem~10.7]{Breiman} that
 \[
\lim_{n\to\infty} \sup_{\genfrac{}{}{0pt}{}{z\in\R^m}{|z|\le L}}
\Big|a_n^m\int_Y \zeta_1\;\zeta_2\circ F^n \; 
1_{\{\tu_n\in z_1a_n+K,\,\hv^{(n)}=[z_2a_n]\}}\,d\mu
- g_{X_\alpha}(z)\mu(\zeta_1)\mu(\zeta_2) |K|\Big|=0.
\]

Next, given a measurable subset $A\subset\cA$ with
$\mu(\partial A)= 0$, we approximate $A$ and $Y\setminus A$ from within as in the proof of Proposition~\ref{prop:A}. In place of cylinders $\beta_n$, we use sets of the form $\beta_n(y)=\{F^ny':s(y,y')>2n\}$. Similarly for $B$. In this way, we obtain that
\begin{align*}
\lim_{n\to\infty} \sup_{\genfrac{}{}{0pt}{}{z\in\R^m}{|z|\le L}}
\Big|a_n^m\mu\big(y\in A\cap  F^{-n}B:\, \tu_n(y) & \in z_1a_n +K,
\, \hv^{(n)}(y)=[z_2a_n]   \big)
\\ & - g_{X_\alpha}(z)\mu(A)\mu(B)|K|\Big|=0
\end{align*}
for $A,B,K$ as in the statement of the Theorem.

Finally, we use the same argument as in Subsection~\ref{sec:pf2} to pass from
$\tu_n$ to
$u_n=\tu_n+\chi\circ F^n-\chi$.
By assumption, $\chi$ restricts to a uniformly continuous function on $A$ and $B$. Hence
a simpler version of the argument in Proposition~\ref{prop:N4} shows that
the functions $H_1=-\chi$ and $H_2=\chi$
satisfy condition~\ref{itm:H}.
We can now apply Lemma~\ref{lem:vague} to deduce the result.
\end{proof}

\paragraph{Acknowledgements}
FP is supported by the ANR projects GALS (ANR-23-CE40-0001) and RAWABRANCH (ANR-23-CE40-0008), and conducted this work within the framework of the Henri Lebesgue Center ANR-11-LABX-0020-01.

We thank the universities of Brest and Leiden for their hospitality.  IM is grateful to the hospitality of Dalian University of Technology where part of this research was carried out.

 \end{document}